\documentclass[12pt,final]{l4dc2021} 

\usepackage{url}
\usepackage{graphicx}
\usepackage{caption}
\usepackage{float}
\usepackage{booktabs}
\usepackage{xcolor}
\usepackage{graphics}
\usepackage{caption}
\usepackage{bm}
\usepackage{epstopdf}
\usepackage{amsmath}
\usepackage{amssymb}
\usepackage{enumerate}
\usepackage{balance}
\usepackage{empheq}
\usepackage{mathtools}

\usepackage{enumitem}
	 
\usepackage{algorithm}
\usepackage{algorithmic,array}
\DeclareMathOperator*{\argmin}{arg\,min}
 \newcommand{\tr}{{{\mathsf T}}}

\DeclareMathOperator*{\st}{subject~to}

\newcommand{\norm}[1]{\left\lVert#1\right\rVert}
\newcommand{\normHinf}[1]{\left\lVert#1\right\rVert_{\infty}}
\newcommand{\hatbf}[1]{\hat{\mathbf{#1}}}

\newcommand*\samethanks[1][\value{footnote}]{\footnotemark[#1]}
\author{Foo\thanks{University of Podunk, Timbuktoo}
\and Bar\samethanks
\and Baz\thanks{Somewhere Else}
\and Bof\samethanks[1]
\and Hmm\samethanks}

\usepackage{times}

\usepackage{hyperref}
\hypersetup{
    colorlinks=true,
    linkcolor=blue,
    citecolor = blue,
    filecolor=blue,
    urlcolor=blue,
}

\mathchardef\Re="023C
\mathchardef\Im="023D

\newtheorem{assumption}{Assumption}
\newtheorem{myTheorem}{Theorem}[section]
\newtheorem{myProposition}{Proposition}[section]
\newtheorem{myRemark}{Remark}[section]
\newtheorem{myLemma}{Lemma}[section]
\newtheorem{myCorollary}{Corollary}[section]

\usepackage[nameinlink]{cleveref}
\crefname{equation}{}{}
\crefname{myTheorem}{Theorem}{Theorems}
\crefname{myCorollary}{Corollary}{Corollaries}
\crefname{assumption}{Assumption}{Assumptions}
\crefname{myLemma}{Lemma}{Lemmas}
\crefname{myRemark}{Remark}{Remarks}
\crefname{myProposition}{Proposition}{Propositions}
\crefname{figure}{Figure}{Figures}
\crefname{table}{Table}{Tables}
\crefname{section}{Section}{Sections}
\crefname{appendix}{Appendix}{Appendices}

\newcommand{\blue}[1]{{\color{blue}#1}}

\newcommand{\preprintswitch}[2]{#2} 

\title[Sample Complexity of LQG Control for Output Feedback Systems]{Sample Complexity of Linear Quadratic Gaussian (LQG) Control \\ for Output Feedback Systems}

	 \author{\Name{Yang Zheng} \thanks{Yang Zheng and Luca Furieri contributed equally to this work.}  \Email{zhengy@g.harvard.edu} \thanks{Yang Zheng and Na Li are supported by NSF career, AFOSR YIP, and ONR YIP.}\\  \addr School of Engineering and Applied Sciences, Harvard University, USA \AND \Name{Luca Furieri}$~^{\blue{\ast}}$   \Email{luca.furieri@epfl.ch} \thanks{Luca Furieri and Maryam Kamgarpour are gratefully supported by ERC Starting Grant CONENE.}\\  \addr  Automatic Control Laboratory, ETH Zurich, Switzerland,\\ Laboratoire d'Automatique, EPFL, Switzerland   \AND \Name{Maryam Kamgarpour} \Email{maryamk@ece.ubc.ca} $^{\blue{\ddagger}}$ \\ \addr  Electrical and Computer Engineering, University of British Columbia, Canada  \AND \Name{Na Li} \Email{nali@seas.harvard.edu} $^{\blue{\dagger}}$  \\  \addr School of Engineering and Applied Sciences, Harvard University, USA    
}

\begin{document}
\maketitle

\maketitle

\begin{abstract}

This  paper studies a class of partially observed Linear Quadratic Gaussian (LQG) problems with \emph{unknown} dynamics. We establish an end-to-end sample complexity bound on learning a robust LQG controller for open-loop stable plants. This is achieved using  a robust synthesis procedure, where we first estimate a model from a single input-output trajectory of finite length, identify an H-infinity bound 
on the estimation error, 
and then design a robust controller using the estimated model and its quantified uncertainty. Our synthesis procedure leverages a recent control tool called Input-Output Parameterization (IOP) that enables robust controller design using convex optimization. 
For open-loop stable systems, we prove that the LQG performance degrades \emph{linearly} with respect to the model estimation error using the proposed synthesis procedure. Despite the hidden states in the LQG problem, the achieved scaling matches  previous results on learning Linear Quadratic Regulator (LQR) controllers with full state observations. 


\end{abstract}

\section{Introduction}

There has been a surging interest in applying machine learning techniques to the control of dynamical systems with continuous action spaces (see e.g.,~\cite{duan2016benchmarking, recht2019tour}). 
An increasing body of recent studies have started to address theoretical and practical aspects of deploying learning-based control policies in dynamical systems~\citep{recht2019tour}. An extended review of related work is given in \preprintswitch{Appendix~A of \cite{ARXIV}}{\cref{app:related_work}}\preprintswitch{\footnote{The report \cite{ARXIV} also contains technical proofs and complementary discussion.}}{}.

For data-driven reinforcement learning (RL) control, the existing algorithmic frameworks can be broadly divided into two categories: (a) model-based RL, in which an agent first fits a model for the system dynamics from observed data and then uses this model to design a policy using either the certainty equivalence principle~\citep{aastrom2013adaptive,tu2018gap,mania2019certainty} or classical robust control tools~\citep{zhou1996robust, dean2019sample,tu2017non}; and (b) model-free RL, in which the agent attempts to learn an optimal policy directly from the data without explicitly building a model for the system~\citep{fazel2018global, malik2018derivative,furieri2019learning}. Another interesting line of work formulates model-free control by using past trajectories to predict future trajectories based on so called \emph{fundamental lemma}~\citep{coulson2019data, berberich2019robust,de2019formulas}. For both model-based and model-free methods, it is critical to establish formal guarantees on their sample efficiency, stability and robustness. 
Recently, the Linear Quadratic Regulator (LQR), one of the most well-studied optimal control problems, has been adopted as a benchmark to understand how machine learning interacts with continuous control~\citep{dean2019sample,tu2018gap,mania2019certainty, fazel2018global, recht2019tour,dean2018regret,malik2018derivative}. 
It was shown that the simple certainty equivalent model-based method requires asymptotically less samples than model-free policy gradient methods for LQR~\citep{tu2018gap}. Besides, the certainty equivalent control~\citep{mania2019certainty} (scaling as ${\mathcal{O}}(N^{-1})$, where $N$ is the number of samples) 
is more sample-efficient than robust model-based methods (scaling as ${\mathcal{O}}(N^{-1/2})$) that account for uncertainty explicitly~\citep{dean2019sample}. 

In this paper, we take a step further towards a theoretical understanding of model-based learning methods for Linear Quadratic Gaussian (LQG) control.  As one of the most fundamental control problems, LQG deals with partially observed linear dynamical systems driven by additive white Gaussian noises~\citep{zhou1996robust}. As a significant challenge compared to LQR, the internal system states cannot be directly measured for learning and control purposes.  
When the system model is known, LQG admits an elegant 
closed-form solution, combining a Kalman filter together with an LQR feedback gain~\citep{bertsekas2011dynamic,zhou1996robust}. For unknown dynamics, however, much fewer results are available for the achievable closed-loop performance. One natural solution is the aforementioned \emph{certainty equivalence principle}: collect some data of the system evolution, fit a model, and then solve the original LQG problem by treating the fitted model as the truth~\citep{aastrom2013adaptive}. It has been recently proved in~\cite{mania2019certainty} that this certainty equivalent principle enjoys a good statistical rate for sub-optimality gap that scales as the \emph{square} of the model estimation error. 
However, this procedure does not come with a robust stability guarantee, and it might fail to stabilize the system when data is not sufficiently large. Sample-complexity of Kalman filters has also been recently characterized in \cite{tsiamis2020sample}. 

Leveraging recent advances in control synthesis~\citep{furieri2019input} and non-asymptotic system identification~\citep{oymak2019non,tu2017non,sarkar2019finite,zheng2020non}, we establish {an end-to-end sample-complexity result} of learning LQG controllers that robustly stabilize the true system 
with a high probability. In particular, our contribution is on developing a novel tractable robust control synthesis procedure, whose sub-optimality 
can be tightly bounded as a function of the model uncertainty. By incorporating a non-asymptotic $\mathcal{H}_\infty$ bound on the system estimation error, we establish an 
end-to-end sample complexity bound for learning robust LQG controllers. \cite{dean2019sample} performed similar analysis for learning LQR controllers with full state measurements. Instead, our method includes noisy output measurements without reconstruting an internal state-space representation for the system. Despite the challenge of hidden states, for open-loop stable systems, our method achieves the same scaling for the sub-optimality gap as~\cite{dean2019sample}, that is, $\mathcal{O}\left(\epsilon\right)$, where $\epsilon$ is the model uncertainty level. 
Specifically, the highlights of our work include:
\begin{itemize}[leftmargin=*]
\setlength\itemsep{0em}
    \item 
    Our design methodology is suitable for general multiple-input multiple-output (MIMO) LTI systems that are open-loop stable. Based on a recent control tool, called Input-Ouput Parameterization (IOP)~\citep{furieri2019input}, we derive a new convex parameterization of robustly stabilizing controllers. Any feasible solution from our procedure corresponds to a controller that is robust against model uncertainty. Our framework directly aims for a class of general LQG problems, going beyond the recent 
    results~\citep{dean2019sample,boczar2018finite} that are built on the system-level parameterization (SLP)~\citep{wang2019system}.
    
    \item We quantify the performance degradation of the robust LQG controller, scaling \emph{linearly} with the model error, which is consistent with~\cite{dean2019sample,boczar2018finite}. Our analysis requires a few involved bounding arguments in the IOP framework~\citep{furieri2019input} due to the absence of direct state measurements. We note that this linear scaling is inferior to the simple certainty equivalence controller~\citep{mania2019certainty}, for which the performance degradation scales as the \emph{square} of parameter errors for both LQR and LQG, but without guarantees on the robust stablility against model errors. 
    This brings an interesting trade-off between optimality and robustness, which is also observed in the LQR case~\citep{dean2019sample}.
    \vspace{-3pt}
\end{itemize}

The rest of this paper is organized as follows. We introduce Linear Quadratic Gaussian (LQG) control for unknown systems and overview our contributions in \cref{section:problemstatement}.  In \cref{section:robustIOP}, we first leverage the IOP framework to develop a robust controller synthesis procedure taking into account estimation errors explicitly, and then derive our main sub-optimality result.  This enables our end-to-end sample complexity analysis discussed in \cref{section:performance}. 
We conclude this paper in \cref{section:Conclusions}. Proofs are postponed to the~appendices\preprintswitch{ of~\cite{ARXIV}}{.} 

\textit{Notation.} 
We use lower and upper case letters (\emph{e.g.} $x$ and $A$) to denote vectors and matrices, and lower and upper case boldface letters (\emph{e.g.} $\mathbf{x}$ and $\mathbf{G}$) are used to denote signals and transfer matrices, respectively. 
%
 Given a stable transfer matrix $\mathbf{G} \in \mathcal{RH}_{\infty}$, where $\mathcal{RH}_{\infty}$ denotes the subspace of stable transfer matrices, we denote its $\mathcal{H}_{\infty}$ norm by $\|\mathbf{G}\|_{\infty}:= \sup_{\omega} {\sigma}_{\max} (\mathbf{G}(e^{j\omega}))$. 

\section{Problem Statement and Our contributions} \label{section:problemstatement}


\subsection{LQG formulation}

We consider the following \emph{partially observed} output feedback system
\vspace{-3pt}
\begin{equation} \label{eq:dynamic}
	\begin{aligned}
	    x_{t+1} &= A_\star x_t+B_\star u_t, \\ y_t &= C_\star x_t + v_t, \\
	    u_t &= \pi(y_t,\ldots,y_0)+w_t\,,
	\end{aligned}    
\end{equation}
where $x_t \in \mathbb{R}^n$ is the state of the system, $u_t\in \mathbb{R}^m$ is the control input and $\pi(\cdot)$ is an output-feedback control policy, $y_t\in \mathbb{R}^p$ is the observed output, and $w_t\in \mathbb{R}^m, v_t\in \mathbb{R}^p$ are Gaussian noise with zero-mean and covariance $\sigma_w^2 I, \sigma_v^2 I$. 
The setup in~\eqref{eq:dynamic} is convenient from an external input-output perspective, where the noise $w_t$ affects the input $u_t$ directly.\footnote{Letting $\hat{w}_t = B_\star w_t$,~\eqref{eq:dynamic} is an instance of the classical LQG formulation where the process noise $\hat{w}_t$ has variance $\sigma_w^2B_\star B_\star^\mathsf{T}$. The setting~\eqref{eq:dynamic}  enables a concise closed-loop representation in~\eqref{eq:responses_maintext}, facilitating the suboptimality analysis in the IOP framework.  
We leave the general case with unstructured covariance for future work.} This setup was also considered in~\cite{tu2017non,boczar2018finite,zheng2019systemlevel}.
 When $C = I, v_t =0$, \emph{i.e.}, the state $x_t$ is directly measured, 
 the system is called \emph{fully observed}. 
Throughout this paper, we make the following assumption. 
\begin{assumption} \label{assumption:stabilizability}
    $(A_\star, B_\star)$ is stabilizable and  $(C_\star, A_\star)$ detectable.
\end{assumption}

The classical Linear Quadratic Gaussian (LQG) control problem is defined as  
\begin{equation} \label{eq:LQG}
    \begin{aligned}
        \min_{u_0,u_1,\ldots} \quad & \lim_{T \rightarrow \infty}\mathbb{E}\left[ \frac{1}{T} \sum_{t=0}^T \left(y_t^\tr Q y_t + u_t^\tr R u_t\right)\right] \\
        \st \quad & ~\eqref{eq:dynamic}, 
    \end{aligned}
\end{equation}
where $Q$ and $R$ are positive definite. 
%
%
    Without loss of generality and for notational simplicity, we assume that $Q=I_p$, $R=I_m$, $\sigma_w=\sigma_v = 1$. 
When the dynamics~\eqref{eq:dynamic} are known, this problem has a well-known closed-form solution by solving two algebraic Riccati equations~\citep{zhou1996robust}. The optimal solution is $u_t = K \hat{x}_t$ with a fixed $p \times n$ matrix $K$ and $\hat{x}_t$ is the state estimation from the observation $y_0, \ldots, y_t$ using the Kalman filter. Compactly, the optimal controller to~\eqref{eq:LQG} can be written 
in the form of transfer function 
$
    \mathbf{u}(z) = \mathbf{K}(z) \mathbf{y}(z),
$ 
with $\mathbf{K}(z) = C_k(zI - A_k)^{-1}B_k + D_k$, where $z\in\mathbb{C}$, $\mathbf{u}(z)$ and $\mathbf{y}(z)$ are the $z-$transform of the input $u_t$ and output $y_t$, and the transfer function $K(z)$ has a state-space realization expressed as
\begin{equation} \label{eq:dyController}
	\begin{aligned}
	    \xi_{t+1}&=A_k\xi_t+B_ky_t,\\
	    u_t&=C_k\xi_t+D_ky_t,
	\end{aligned}
\end{equation}
where $\xi \in \mathbb{R}^q$ is the controller internal state, 
and $A_k, B_k, C_k, D_k$ depends on the system matrices $A, B, C$ and solutions to algebraic Riccati equations. We refer interested readers to~\cite{zhou1996robust,bertsekas2011dynamic} for more details. 

Throughout the paper, we make another assumption, which is required in both plant estimation algorithms~\citep{oymak2019non} and the robust controller synthesis phase.

\begin{assumption} \label{assumption:open_loop_stability}
    The plant dynamics are open-loop stable, \emph{i.e.}, $\rho(A_\star) < 1$, where $\rho(\cdot)$ denotes the spectral radius. 
\end{assumption}

\subsection{LQG for unknown dynamics}
In the case where the system dynamics $A_\star, B_\star, C_\star$ are unknown, 
one natural idea is to conduct experiments to estimate $\hat{A}, \hat{B}, \hat{C}$~\citep{ljung2010perspectives} and to design the corresponding controller based on the estimated dynamics, which is known as \emph{certainty equivalence control}. When the estimation is accurate enough, the certainty equivalent controller leads to good closed-loop performance~\citep{mania2019certainty}. 
However, the certainty equivalent controller does not take into account estimation errors, which might lead to instability in practice. 
It is desirable to explicitly incorporate the estimation errors 
\begin{equation} \label{eq:estimation_error_state_space}
\|\hat{A} - A_\star\|,\quad \|\hat{B} - B_\star\|, \quad \|\hat{C} - C_\star\|, 
\end{equation}
into the controller synthesis, and this requires novel tools from robust control. 
Unlike the fully observed LQR case~\citep{dean2019sample}, 
in the partially observed case of \eqref{eq:dynamic}, it remains unclear how to 
directly 
 incorporate state-space model errors~\eqref{eq:estimation_error_state_space} into robust controller synthesis. Besides, the state-space realization of a partially observed system is not unique, and different realizations from the estimation procedure might have an impact on the controller synthesis.   

Instead of the state-space form~\eqref{eq:dynamic}, the system dynamics can be described uniquely in the frequency domain in terms of the transfer function as 
$$
    \mathbf{G}_\star(z) = C_\star (zI - A_\star)^{-1} B_\star\,, 
$$
where $z\in\mathbb{C}$.
%
%
%
%
Based on an estimated model $\hatbf{G}$ and an upper bound $\epsilon$ on its estimation error $\|\bm{\Delta}\|_\infty :=\|\mathbf{G}_\star-\hatbf{G}\|_{\infty}$, we consider a robust variant of the LQG problem that seeks to minimize the worst-case LQG performance of the closed-loop system 
\begin{equation} \label{eq:robustLQG}
    \begin{aligned}
        \min_{\mathbf{K}} \sup_{\|\mathbf{\Delta}\|_{\infty} < 
        \epsilon} \quad & \lim_{T \rightarrow \infty}\mathbb{E}\left[ \frac{1}{T} \sum_{t=0}^T \left(y_t^\tr Q  y_t + u_t^\tr R u_t\right)\right], \\
        \st \quad & \mathbf{y} = (\hat{\mathbf{G}} + \mathbf{\Delta}) \mathbf{u} + \mathbf{v}\,, \quad \mathbf{u} = \mathbf{K} \mathbf{y} + \mathbf{w},
    \end{aligned}
\end{equation}
where $\mathbf{K}$ is a proper transfer function. When $\bm{\Delta} = 0$, \eqref{eq:robustLQG} recovers the standard LQG formulation \eqref{eq:LQG}. Additionally, note that the cost of \eqref{eq:robustLQG} is finite if and only if $\mathbf{K}$ stabilizes all plants $\mathbf{G}$ such that $\mathbf{G} = \hatbf{G} + \bm{\Delta}$ and $\norm{\bm{\Delta}}_\infty<\epsilon$. 

\subsection{Our contributions} 
Although classical approaches exist to compute controllers that stabilize all plants $\hatbf{G}$ with  $\norm{\bm{\Delta}}_\infty \leq \epsilon\,,$~\citep{zhou1996robust}, these methods typically do not quantify the closed-loop LQG performance degradation in terms of the uncertainty size $\epsilon$. In this paper, we exploit the recent IOP framework~\citep{furieri2019input} to develop a tractable inner approximation of \eqref{eq:robustLQG} (see \cref{theo:mainRobust}). 
In our main \cref{theo:suboptimality}, if $\epsilon$ is small enough, we bound the suboptimality performance gap as 
\begin{equation*}
\frac{\hat{J} - J_\star}{J_\star} 
     \leq  
     M \epsilon\,,
\end{equation*}
where $J_\star$ is the globally optimal LQG cost to~\eqref{eq:LQG}, and $\hat{J}$ is the LQG cost when applying the robust controller from our procedure to the true plant $\mathbf{G}_{\star}$, $\epsilon > \|\mathbf{G}_\star-\hat{\mathbf{G}}\|_\infty$ is an upper bound of estimation error, and $M$ is a constant that depends explicitly on true dynamics $\mathbf{G}_\star$, its estimation $\hatbf{G}$, and the true LQG controller $\mathbf{K}_
{\star}$; see \eqref{eq:suboptimality} for a precise expression. 
Adapting recent non-asymptotic estimation results from input-output trajectories~\citep{oymak2019non,sarkar2019finite,tu2017non,zheng2020non}, we derive an end-to-end sample complexity of learning LQG controllers as
\begin{equation*}
    \frac{\hat{J} - J_\star}{J_\star} 
     \leq {\mathcal{O}}\left(\frac{T}{\sqrt{N}}+\rho(A_\star)^T\right)\,,
\end{equation*}
with high probability provided $N$ is sufficiently large, where $T$ is the length of finite impulse response (FIR) model estimation, and $N$ is the number of samples in input-output trajectories (see \cref{corollary:samplecomplexity}). 
When the true plant $\mathbf{G}_{\star}$ is an FIR, the sample complexity scales as ${\mathcal{O}}\left(N^{-1/2}\right)$.

If $\mathbf{G}_\star$ is unstable, the residual $\bm{\Delta} = \mathbf{G}_\star-\hatbf{G}$ might be unstable as well, and thus $\normHinf{\bm{\Delta}}=\infty$. Instead, $\bm{\Delta}$ is always a stable residual when $\mathbf{G}_\star$ is stable, and thus  $\normHinf{\bm{\Delta}}$ is finite.    Also, it is hard to utilize a single trajectory for identifying unstable systems. Finally, unstable residues in the equality constraints of IOP pose a challenge for controller implementation; we refer to Section 6 of \cite{zheng2019systemlevel} for details. We therefore leave the case of unstable systems for future work.

\section{Robust controller synthesis} \label{section:robustIOP}


We first derive a tractable convex approximation for~\eqref{eq:robustLQG} using the recent input-output parameterization (IOP) framework~\citep{furieri2019input}. This allows us to compute a robust controller using convex optimization. We then provide sub-optimality guarantees in terms of the uncertainty size $\epsilon$. {The overall principle is parallel to that of~\cite{dean2019sample} for the LQR case.} 

\subsection{An equivalent IOP reformulation of \eqref{eq:robustLQG}}

Similar to the Youla parameterization~\citep{youla1976modern} and the SLP~\citep{wang2019system}, the IOP framework~\citep{furieri2019input} focuses on the \emph{system responses} of a closed-loop system. In particular, given an arbitrary control policy $\mathbf{u} = \mathbf{K} \mathbf{y}$, straightforward calculations show that the closed-loop responses from the noises $\mathbf{v}, \mathbf{w}$ to the output $\mathbf{y}$ and control action $\mathbf{u}$ are 
\begin{equation} \label{eq:responses_maintext}
    \begin{bmatrix} \mathbf{y} \\ \mathbf{u} \end{bmatrix} = \begin{bmatrix}(I-\mathbf{G}_\star\mathbf{K})^{-1}&(I-\mathbf{G}_\star\mathbf{K})^{-1}\mathbf{G}_\star\\\mathbf{K}(I-\mathbf{G}_\star\mathbf{K})^{-1}&(I-\mathbf{K}\mathbf{G}_\star)^{-1}\end{bmatrix} \begin{bmatrix} \mathbf{v} \\ \mathbf{w} \end{bmatrix}.
\end{equation}
To ease the notation, we can define the closed-loop responses as 
\begin{equation} \label{eq:YUWZ_main_text}
     \begin{bmatrix} \mathbf{Y} & \mathbf{W}\\\mathbf{U}&\mathbf{Z}\end{bmatrix} := \begin{bmatrix}(I-\mathbf{G}_\star\mathbf{K})^{-1}&(I-\mathbf{G}_\star\mathbf{K})^{-1}\mathbf{G}_\star\\\mathbf{K}(I-\mathbf{G}_\star\mathbf{K})^{-1}&(I-\mathbf{K}\mathbf{G}_\star)^{-1}\end{bmatrix}.
\end{equation}
Given any stabilizing $\mathbf{K}$, we can write $\mathbf{K}=\mathbf{UY}^{-1}$ and the square root of the LQG cost in~\eqref{eq:LQG} as 
\begin{equation}
\label{eq:cost_definition_convex_v2}
J(\mathbf{G}_{\star},\mathbf{K})=\left\|\begin{bmatrix}\mathbf{Y}&\mathbf{W}\\ \mathbf{U}&\mathbf{Z}\end{bmatrix}\right\|_{\mathcal{H}_2},
\end{equation}
with the closed-loop responses $(\mathbf{Y},\mathbf{U},\mathbf{W},\mathbf{Z})$ defined in~\eqref{eq:YUWZ_main_text}; see~\cite{furieri2019input,zheng2019equivalence} for more details on IOP. We present a full equivalence derivation of~\eqref{eq:cost_definition_convex_v2} in \preprintswitch{\cite[Appendix~G]{ARXIV}}{\cref{app:h2norm}}.
%
%

Our first result is a reformulation of the robust LQG problem~\eqref{eq:robustLQG} in the IOP framework.
\begin{myTheorem}
\label{pr:robust_LQG}
    The robust LQG problem~\eqref{eq:robustLQG} is equivalent to 
    \begin{align}
        \label{eq:robustLQG_YUWZ}
        \min_{\hat{\mathbf{Y}}, \hat{\mathbf{W}}, \hat{\mathbf{U}}, \hat{\mathbf{Z}}} \; \max_{\|\mathbf{\Delta}\|_\infty< \epsilon} \quad &      J(\mathbf{G}_{\star},\mathbf{K}) = \left\|\begin{bmatrix}\hat{\mathbf{Y}}(I-\mathbf{\Delta}\hat{\mathbf{U}})^{-1}&\hat{\mathbf{Y}}(I-\mathbf{\Delta}\hat{\mathbf{U}})^{-1}(\hat{\mathbf{G}}+\mathbf{\Delta})\\\hat{\mathbf{U}}(I-\mathbf{\Delta}\hat{\mathbf{U}})^{-1}&(I-\hat{\mathbf{U}}\mathbf{\Delta})^{-1}\hat{\mathbf{Z}}\end{bmatrix}\right\|_{\mathcal{H}_2}  \\
        \st \quad &  
        \begin{bmatrix}
        	 I&-\hat{\mathbf{G}}
        	 \end{bmatrix}\begin{bmatrix}
        	 \hat{\mathbf{Y}} & \hat{\mathbf{W}}\\\hat{\mathbf{U}} & \hat{\mathbf{Z}}
        	 \end{bmatrix}=\begin{bmatrix}
        	 I&0
        	 \end{bmatrix}, \label{eq:ach_hat1}\\
        &	 \begin{bmatrix}
        	  \hat{\mathbf{Y}} & \hat{\mathbf{W}}\\\hat{\mathbf{U}} & \hat{\mathbf{Z}}
        	 \end{bmatrix}
        	 \begin{bmatrix}
        	 -\hat{\mathbf{G}}\\I
        	 \end{bmatrix}=\begin{bmatrix}
        	 0\\I
        	 \end{bmatrix},  \label{eq:ach_hat2}\\
        &	 \hat{\mathbf{Y}}, \hat{\mathbf{W}}, \hat{\mathbf{U}}, \hat{\mathbf{Z}}\in \mathcal{RH}_\infty, \, \|\hat{\mathbf{U}}\|_{\infty} \leq \frac{1}{\epsilon}, \nonumber
    \end{align}

  where the optimal robust controller is recovered from the optimal $\hatbf{U}$ and $\hatbf{Y}$ as $\mathbf{K} = \hat{\mathbf{U}}\hat{\mathbf{Y}}^{-1}.$
\end{myTheorem}

The proof relies on a novel robust variant of IOP for parameterizing robustly stabilizing controller in a convex way. We provide the detailed proof and a review of the IOP framework in \preprintswitch{\cite[Appendix~C]{ARXIV}}{\cref{app:robuststability}}. Here, it is worth noting that the feasible set in \eqref{eq:robustLQG_YUWZ} is convex in the decision variables $(\hatbf{Y},\hatbf{W},\hatbf{U},\hatbf{Z})$, which represent four closed-loop maps on the estimated plant $\hat{\mathbf{G}}$. Using the small-gain theorem \citep{zhou1996robust}, the convex requirement $\|\hat{\mathbf{U}}\|_{\infty} \leq \frac{1}{\epsilon}$ ensures that any controller $\mathbf{K} = \hat{\mathbf{U}}\hat{\mathbf{Y}}^{-1}$, with $\hatbf{Y},\hatbf{U}$ feasible for \eqref{eq:robustLQG_YUWZ}, stabilizes the real plant $\mathbf{G}_\star$ for all $\bm{\Delta}$ such that $\|\bm{\Delta}\|_\infty < \epsilon$.

Due to the uncertainty $\bm{\Delta}$, the cost in~\eqref{eq:robustLQG_YUWZ}  is nonconvex in the decision variables. We therefore proceed with deriving an upper-bound on the functional $ J(\mathbf{G}_\star, \mathbf{K})$, which will be exploited to derive a quasi-convex approximation of the robust LQG problem~\eqref{eq:robustLQG}.


\subsection{Upper bound on the non-convex cost in~\eqref{eq:robustLQG_YUWZ}}
It is easy to derive (see \preprintswitch{Appendix~B of   \cite{ARXIV}}{\cref{app:inequalities}}) that
\begin{equation}
\label{eq:LQGbound_step1}
\begin{aligned}
     J(\mathbf{G}_\star,\mathbf{K})^2 =  \|\hat{\mathbf{Y}}(I-\mathbf{\Delta}\hat{\mathbf{U}})^{-1}&\|_{\mathcal{H}_2}^2 + \|\hat{\mathbf{U}}(I-\mathbf{\Delta}\hat{\mathbf{U}})^{-1}\|_{\mathcal{H}_2}^2  \\
     & + \|(I-\hat{\mathbf{U}}\mathbf{\Delta})^{-1}\hat{\mathbf{Z}}\|_{\mathcal{H}_2}^2 + \|\hat{\mathbf{Y}}(I-\mathbf{\Delta}\hat{\mathbf{U}})^{-1}(\hat{\mathbf{G}}+\mathbf{\Delta}) \|_{\mathcal{H}_2}^2.
\end{aligned}
\end{equation}
{Similarly to \cite{dean2019sample} for the LQR case, it is relatively easy to bound the first three terms on the right hand side of~\eqref{eq:LQGbound_step1} using small-gain arguments. However, dealing with outputs makes it challenging to bound the last term. The corresponding result is summarized in the following proposition (see \preprintswitch{Appendix~D.1 of   \cite{ARXIV}}{\cref{app:upperbound_lastterm}} for proof).}

\begin{myProposition} \label{prop:LQGcostupperbound_step1}
    If $\|\hat{\mathbf{U}}\|_{\infty} < \frac{1}{\epsilon}, \|\mathbf{\Delta}\|_\infty < \epsilon$ and $\hat{\mathbf{G}}\in \mathcal{RH}_{\infty}$, then, we have  
    \begin{equation} \label{eq:LQGbound_step3}
    \begin{aligned}
       \|\hat{\mathbf{Y}}(I-\mathbf{\Delta}\hat{\mathbf{U}})^{-1}(\hat{\mathbf{G}}+\mathbf{\Delta})\|_{\mathcal{H}_2} &  \leq  \frac{\|\hat{\mathbf{W}}\|_{\mathcal{H}_2} + \epsilon \|\hat{\mathbf{Y}}\|_{\mathcal{H}_2} (2 + \|\hat{\mathbf{U}}\|_{\infty}\|\hat{\mathbf{G}}\|_\infty )}{1 - \epsilon \|\hat{\mathbf{U}}\|_{\infty}},\\
    \end{aligned}
\end{equation}
where $\hat{\mathbf{W}} = \hat{\mathbf{Y}}\hat{\mathbf{G}}$. 
\end{myProposition}

We are now ready to present an upper bound on the LQG cost. The proof is based on \cref{prop:LQGcostupperbound_step1} and basic inequalities; see  \preprintswitch{\cite[Appendix~D.2]{ARXIV}}{\cref{app:prop_LQGcostupperbound}}.

\begin{myProposition} \label{prop:LQGcostupperbound}
    If $\|\hat{\mathbf{U}}\|_{\infty} \leq \frac{1}{\epsilon}, \|\mathbf{\Delta}\|_\infty < \epsilon$ and $\hat{\mathbf{G}}\in \mathcal{RH}_{\infty}$, the robust LQG cost in~\eqref{eq:robustLQG_YUWZ} is upper bounded by 
    \begin{equation}
    \label{eq:cost_upperbound_nonconvex}
        J(\mathbf{G}_\star,\mathbf{K}) \leq \frac{1}{1 - \epsilon \|\hat{\mathbf{U}}\|_{\infty}} \left\|\begin{bmatrix} \sqrt{1 + h(\epsilon,\|\hat{\mathbf{U}}\|_{\infty})} \hat{\mathbf{Y}} & \hat{\mathbf{W}}\\
        \hat{\mathbf{U}}& \hat{\mathbf{Z}}\end{bmatrix}\right\|_{\mathcal{H}_2},
    \end{equation}
    where $\hat{\mathbf{Y}}, \hat{\mathbf{W}}, \hat{\mathbf{U}}, \hat{\mathbf{Z}}$ satisfy the constraints in~\eqref{eq:robustLQG_YUWZ}, and the factor $h(\epsilon,\|\hat{\mathbf{U}}\|_{\infty})$ is defined as
    \begin{equation} \label{eq:factor_h}
        h(\epsilon,\|\hat{\mathbf{U}}\|_{\infty}) :=     \epsilon\|\hat{\mathbf{G}}\|_{\infty} (2 + \|\hat{\mathbf{U}}\|_{\infty}\|\hat{\mathbf{G}}\|_\infty) + \epsilon^2 (2 + \|\hat{\mathbf{U}}\|_{\infty}\|\hat{\mathbf{G}}\|_\infty)^2.  
    \end{equation}
\end{myProposition}


\subsection{Quasi-convex formulation}

Building on the LQG cost upper bound \eqref{eq:cost_upperbound_nonconvex}, we derive our first main result on a tractable approximation of \eqref{eq:robustLQG_YUWZ}. The proof is reported in \preprintswitch{\cite[Appendix~D.3]{ARXIV}}{\cref{app:th_approximation}}.

\begin{myTheorem} \label{theo:mainRobust}
   Given $\hat{\mathbf{G}} \in \mathcal{RH}_\infty$, a model estimation error $\epsilon$, and any constant {$\alpha>0$}, the robust LQG problem~\eqref{eq:robustLQG_YUWZ} is upper bounded by the following problem
     \begin{equation} \label{eq:robustLQG_convex}
    \begin{aligned}
      \min_{\gamma \in[0,{1}/{\epsilon})} \frac{1}{1-\epsilon \gamma} \;\;  \min_{\hat{\mathbf{Y}}, \hat{\mathbf{W}}, \hat{\mathbf{U}}, \hat{\mathbf{Z}}} \; &\left\|\begin{bmatrix} \sqrt{1 + h(\epsilon, \alpha)} \hat{\mathbf{Y}} & \hat{\mathbf{W}}\\
        \hat{\mathbf{U}}& \hat{\mathbf{Z}}\end{bmatrix}\right\|_{\mathcal{H}_2} \\
        \st \quad        & \eqref{eq:ach_hat1}-\eqref{eq:ach_hat2},~	 \hat{\mathbf{Y}}, \hat{\mathbf{W}}, \hat{\mathbf{Z}}\in \mathcal{RH}_\infty, \, \|\hat{\mathbf{U}}\|_{\infty} \leq \min\left(\gamma,\alpha\right),
    \end{aligned}
\end{equation}
where 
$
    h(\epsilon,\alpha) =  \epsilon\|\hat{\mathbf{G}}\|_{\infty} (2 + \alpha\|\hat{\mathbf{G}}\|_\infty) + \epsilon^2 (2 + \alpha\|\hat{\mathbf{G}}\|_\infty)^2.  
$
\end{myTheorem}

The hyper-parameter $\alpha$ in~\eqref{eq:robustLQG_convex} plays two important roles: 1) \emph{robust stability}: the resulting controller has a guarantee of robust stability against model estimation error up to $ \frac{1}{\alpha}$, {thus suggesting $\alpha < \frac{1}{\epsilon}$, as we will clarify it later}; 2) \emph{quasi-convexity}: the inner optimization problem~\eqref{eq:robustLQG_convex} is convex when fixing $\gamma$, and the outer optimization  is quasi-convex with respect to $\gamma$, which can effectively be solved using the golden section search.  

\begin{myRemark}[Numerical computation]
  \begin{enumerate}
      \item The inner optimization in~\eqref{eq:robustLQG_convex} is convex but infinite dimensional. A practical numerical approach is to apply a finite impulse response (FIR) truncation on the decision variables  $\hat{\mathbf{Y}}, \hat{\mathbf{U}}, \hat{\mathbf{W}}, \hat{\mathbf{Z}}$, which leads to a finite dimensional convex semidefinite program (SDP) for each fixed value of $\gamma$; see  \preprintswitch{Appendix~B of   \cite{ARXIV}}{\cref{app:inequalities}}. The degradation in performance decays exponentially with the FIR horizon \citep{dean2019sample}.
      \item Since $\hat{\mathbf{G}} \in \mathcal{RH}_\infty$ is stable, the IOP framework is numerically robust \citep{zheng2019systemlevel}, \emph{i.e.}, the resulting controller $\mathbf{K} = \hat{\mathbf{U}}\hat{\mathbf{Y}}^{-1}$ is stabilizing even when numerical solvers induce small computational residues in~\eqref{eq:robustLQG_convex}.
  \end{enumerate}
   
\end{myRemark}

\subsection{Sub-optimality guarantee} 
%


Our second main result offers a sub-optimality guarantee on the performance of the robust controller synthesized using the robust IOP framework~\eqref{eq:robustLQG_convex} in terms of the estimation error $\epsilon$. The proof is reported in \preprintswitch{\cite[Appendix~E]{ARXIV}}{\cref{app:theo:suboptimality}}.

\begin{myTheorem} \label{theo:suboptimality}
Let $\mathbf{K}_\star$ be the optimal LQG controller in~\eqref{eq:LQG}, and the corresponding closed-loop responses be $\mathbf{Y}_\star, \mathbf{U}_\star, \mathbf{W}_\star,\mathbf{Z}_\star$. Let $\hat{\mathbf{G}}$ be the plant estimation with error $\|\mathbf{\Delta}\|_\infty < \epsilon$, where $\mathbf{\Delta} = \mathbf{G}_\star - \hat{\mathbf{G}}$. Suppose that $\epsilon  < \frac{1}{5\|\mathbf{U}_\star\|_\infty}$, and choose 
the constant hyper-parameter $\alpha \in \left[\frac{ \sqrt{2}\|\mathbf{U}_\star\|_\infty}{1-\epsilon \|\mathbf{U}_\star\|_\infty}, \frac{1}{\epsilon}\right)$. 
We denote the optimal solution to~\eqref{eq:robustLQG_convex} as $\gamma_\star, \hat{\mathbf{Y}}_\star, \hat{\mathbf{U}}_\star, \hat{\mathbf{W}}_\star, \hat{\mathbf{Z}}_\star$. Then, the controller ${\mathbf{K}} =\hat{\mathbf{U}}_\star\hat{\mathbf{Y}}_\star^{-1}$ stabilizes the true plant $\mathbf{G}_\star$ and the relative LQG error is upper bounded by  
\begin{equation} \label{eq:suboptimality}
     \frac{J(\mathbf{G}_\star, {\mathbf{K}})^2 - J(\mathbf{G}_\star, {\mathbf{K}}_\star)^2}{J(\mathbf{G}_\star, {\mathbf{K}}_\star)^2} 
     \leq  
     20 \epsilon \|\mathbf{U}_\star\|_\infty  
     + h(\epsilon,\alpha)+g(\epsilon, \|\mathbf{U}_\star\|_\infty),
\end{equation}
where $ h(\cdot , \cdot)$ is defined in~\eqref{eq:factor_h}  and 
\begin{equation} \label{eq:constant_g}
\begin{aligned}
    g(\epsilon,\|\mathbf{U}_\star\|_{\infty}) = 
 \epsilon\|{\mathbf{G}}_\star\|_{\infty} (2 + \|\mathbf{U}_\star\|_{\infty}\|{\mathbf{G}}_\star\|_\infty) + \epsilon^2 (2 + \|\mathbf{U}_\star\|_{\infty}\|{\mathbf{G}}_\star\|_\infty)^2.
\end{aligned}
\end{equation}
\end{myTheorem}

\cref{theo:suboptimality} shows that the relative error in the LQG cost grows as $\mathcal{O}(\epsilon)$  as long as $\epsilon$ is sufficiently small, and in particular $\epsilon  < \frac{1}{5\|\mathbf{U}_\star\|_\infty}$. 
Previous results in~\cite{dean2019sample} proved a similar convergence rate of $\mathcal{O}(\epsilon)$ for LQR using a robust synthesis procedure based on the SLP~\citep{wang2019system}. Our robust synthesis procedure using the IOP framework extends~\cite{dean2019sample} to a class of LQG problems. Note that our bound is valid for open-loop stable plants, while the method~\cite{dean2019sample} works for all systems at the cost of requiring direct state observations. 
Similar to~\cite{dean2019sample} and related work, the hyper-parameters $\epsilon$ and $\alpha$ have to be tuned in practice without knowing $\mathbf{U}_\star$. 

\begin{myRemark}[Optimality versus Robustness]
    Note that recent results in~\cite{mania2019certainty} show that the certainty equivalent controller achieves a better sub-optimality scaling of $\mathcal{O}(\epsilon^2)$ for both fully observed LQR and partially observed LQG settings, at the cost of a much stricter requirement on admissible uncertainty $\epsilon$. Quoting from \cite{mania2019certainty}, ``\emph{the price of obtaining a faster rate for LQR is that the certainty equivalent controller becomes less robust to model uncertainty.}'' Our result in \cref{theo:suboptimality} shows that this trade-off may hold true for the LQG problem as well. 
    
   \preprintswitch{}{ A classical result from~\cite{doyle1978guaranteed} states that there are no \emph{a priori} guaranteed gain margin for the optimal LQG controller $\mathbf{K}_\star$, which means that $\mathbf{K}_\star$ might fail to stabilize a plant with a small perturbation from $\mathbf{G}_\star$. In contrast, any feasible solution from~\eqref{eq:robustLQG_convex} can robustly stabilize any system $\hat{\mathbf{G}}$ with $\|\hat{\mathbf{G}} - \mathbf{G}_\star\|_\infty < \epsilon$. Thus, especially when a model is unavailable, it might be better to trade some performance with robustness by directly considering $\epsilon$ 
   when solving~\eqref{eq:robustLQG_convex}. } 
\end{myRemark}

\preprintswitch{}{\begin{myRemark}[Open-loop stability]
We note that the upper bound on the suboptimality gap in~\eqref{eq:suboptimality} depends explicitly on the optimal closed-loop performance $\|\mathbf{U}_\star\|_\infty$, open-loop plant dynamics $\|\mathbf{G}_\star\|_\infty$, as well as the estimated plant dynamics $\|\hat{\mathbf{G}}_\star\|_\infty$. If the open-loop system is close to the stability margin, as measured by a larger value of $\|\mathbf{G}_\star\|_\infty$, then the sub-optimality gap might be larger using our robust controller synthesis procedure. This suggests that it might be harder to design a robust controller with small suboptimality if the open-loop system is closer to be unstable. This difficulty is also reflected in the plant estimation procedure; see~\cite{oymak2019non,simchowitz2019learning,zheng2020non} for more discussions. 
It would be very interesting to establish a lower bound for the robust synthesis procedure, which allows us to establish fundamental limits with respect to system instability.  
\end{myRemark}}

\section{Sample complexity} \label{section:performance}


Based on our main results in the previous section, here we discuss how to estimate the plant $\mathbf{G}_\star$, provide a non-asymptotic $\mathcal{H}_\infty$ bound on the estimation error, and finally establish an end-to-end sample complexity of learning LQG controllers.  
By \cref{assumption:open_loop_stability}, \emph{i.e.}, $\mathbf{G}_\star \in \mathcal{RH}_\infty$, we can write  
\begin{equation} \label{eq:plantFIR}
 \mathbf{G}_\star(z) = \sum_{i=0}^\infty \frac{1}{z^i}G_{\star,i} =\sum_{i=0}^{T-1} \frac{1}{z^i}G_{\star,i} + \sum_{i=T}^\infty \frac{1}{z^i}G_{\star,i},     
\end{equation}
where $G_{\star,i} \in \mathbb{R}^{p \times m}$ denotes the $i$-th spectral component. 
Given the state-space representation~\eqref{eq:dynamic}, we have 
$
G_{\star,0}=0\,, G_{\star,i} = C_\star A^{i-1}_\star B_\star,~ \forall i\geq 1\,.
$
As $\rho(A_\star) < 1$, the spectral component $G_{\star,i}$ decays exponentially. Thus, we can use a finite impulse response (FIR) truncation of order $T$ for $\mathbf{G}_\star$:
$$
G_\star =  \begin{bmatrix} 0 & C_\star B_\star & \cdots & C_\star A_\star^{T-2} B_\star\end{bmatrix} \in \mathbb{R}^{p \times Tm}.
$$

Many recent algorithms have been proposed to estimate $G_\star$, e.g.,~\cite{oymak2019non,tu2017non,sarkar2019finite,zheng2020non}, and these algorithms differ from the estimation setup; see \cite{zheng2020non} for a comparison. All of them can be used to establish an end-to-end sample complexity.  Here, we discuss a recent ordinary least-squares (OLS) algorithm~\citep{oymak2019non}. This OLS estimator is based on collecting a trajectory $\{y_t,u_t\}_{t=0}^{T+N-1}$, where $u_t$ is Gaussian with variance $\sigma_u^2I$ for every $t$. The OLS algorithm details are provided in \preprintswitch{\citet[Appendix~F.1]{ARXIV}}{\cref{app:theo_FIR_identification_procedure}}.  
From the  OLS solution $\hat{G}$, we form the estimated plant 
\begin{equation} \label{eq:FIRestimation}
    \hat{\mathbf{G}} := \sum_{k=0}^{T-1} \frac{1}{z^k} \hat{G}_k .
\end{equation}
%
To bound the $\mathcal{H}_\infty$ norm of the estimation error 
$\mathbf{\Delta} := \mathbf{G}_\star - \hat{\mathbf{G}}$, we define 
$
    \Phi(A_\star) = \sup_{\tau \geq 0} \; \frac{\|A_\star^\tau\|}{\rho(A_\star)^\tau}.
$
We  assume $\Phi(A_\star)$ is finite~\citep{oymak2019non}. 
We start from \cite[Theorem 3.2]{oymak2019non} to derive two {new }corollaries. 
The proofs rely on {connecting the} spectral radius of $\hat{G}-G_\star$ with the $\mathcal{H}_\infty$ norm of $\hat{\mathbf{G}} - \mathbf{G}_{\star}$; see \preprintswitch{\citet[Appendix~F.2]{ARXIV}}{\cref{app:theo_FIR_identification_bound}} for details.

\begin{myCorollary}\label{co:FIRestimation}
Under the OLS estimation setup in Theorem 3.2 of \cite{oymak2019non}, {with high probability}, the FIR estimation $\hat{\mathbf{G}}$ in~\eqref{eq:FIRestimation} satisfies
\begin{equation}
\label{eq:sysID_bound}
    {\|\mathbf{G}_\star - \hat{\mathbf{G}}\|_{{\infty}}\leq  \frac{R_w + R_v + R_e}{\sigma_u}\sqrt{\frac{T}{{N}}} + \Phi(A_\star) \norm{C_\star} \norm{B_\star} \frac{\rho(A_\star)^T}{1-\rho(A_\star)}}.
\end{equation}
where $N$ is the length of one input-output trajectory, and $R_w, R_v, R_e \in \mathbb{R}$ are problem-dependent\footnote{See~\cite{oymak2019non} for precise formula and probabilistic guarantees. Note that the dynamics in~\eqref{eq:dynamic} are slightly different from~\cite{oymak2019non}, with an extra matrix $B_\star$ in front of $w_t$. By replacing the matrix $F$ in~\cite{oymak2019non} with $G_\star$, all the analysis and bounds stay the same.}.
\end{myCorollary}


\begin{myCorollary} \label{corollary:FIRestimation}
    Fix an $\epsilon > 0$. Let the length of FIR truncation {satisfy} 
    \begin{equation} \label{eq:FIRlength}
        T > \frac{1}{\log \left(\rho(A_\star)\right)} \log \frac{\epsilon (1 - \rho(A_\star))}{2 \Phi(A_\star)\|C_\star\|\|B_\star\|}.
    \end{equation}
    Under the OLS estimation setup in Theorem 3.2 of \cite{oymak2019non}, and further letting
    \begin{equation} \label{eq:Bound_N_v1}
        N > 
        \max\left\{\frac{4T}{\sigma_u^2 \epsilon^2}(R_w+R_v+R_e)^2, cTm \log^2{(2Tm)}\log^2{(2Nm)} \right\},
    \end{equation}
    {with high probability}, the FIR estimation $\hat{\mathbf{G}}$ in~\eqref{eq:FIRestimation} satisfies
    $\|\mathbf{G}_\star - \hat{\mathbf{G}}\|_\infty < \epsilon$. 
\end{myCorollary}

The lower bound on the FIR length $T$ in~\eqref{eq:FIRlength} guarantees that the FIR truncation error is less than $\epsilon/2$, while the lower bound on $N$ in~\eqref{eq:Bound_N_v1} makes sure that the FIR approximation part is less than  $\epsilon/2$, thus leading to the desired bound with high probability. We note that the terms $R_w, R_v, R_e$ depend on the system dimensions and on the FIR length as ${\mathcal{O}}(\sqrt{T}(p+m+n))$. \cref{corollary:FIRestimation} states that the number of time steps to achieve identification error $\epsilon$ in $\mathcal{H}_\infty$ norm scales as ${\mathcal{O}}(T^2/\epsilon^2)$. 

We finally give an end-to-end guarantee by 
combining \cref{corollary:FIRestimation} with \cref{theo:suboptimality}: 

\begin{myCorollary} \label{corollary:samplecomplexity}
    {Let $\mathbf{K}_\star$ be the optimal LQG controller to~\eqref{eq:LQG}, and the corresponding closed-loop responses be $\mathbf{Y}_\star, \mathbf{U}_\star, \mathbf{W}_\star,\mathbf{Z}_\star$. Choose an estimation error $0 < \epsilon < \frac{1}{5\|\mathbf{U}_\star\|_\infty}$, and the hyper-parameter $\alpha \in \left[\frac{ \sqrt{2}\|\mathbf{U}_\star\|_\infty}{1-\epsilon \|\mathbf{U}_\star\|_\infty}, \frac{1}{\epsilon}\right)$. Estimate $\hat{\mathbf{G}}$~\eqref{eq:FIRestimation} with a trajectory of length $N$ in~\eqref{eq:Bound_N_v1} and an FIR truncation length $T$ in~\eqref{eq:FIRlength}. Then, with high probability, the robust controller $\mathbf{K}$ from~\eqref{eq:robustLQG_convex}
    yields a relative error in the LQG cost satisfying~\eqref{eq:suboptimality}.}
\end{myCorollary}


Since the bound on the {trajectory length} $N$ in~\eqref{eq:Bound_N_v1} scales as $\tilde{\mathcal{O}}(\epsilon^{-2})$ (where the logarithmic factor comes from the FIR length $T$),  the suboptimality gap on the LQG cost roughly scales as $\tilde{\mathcal{O}}\left(\frac{1}{\sqrt{N}}\right)$ when the FIR length $T$ is chosen large-enough accordingly. In particular, when the true plant is FIR of order $\overline{T}$ and $T\geq \overline{T}$, via combining \cref{co:FIRestimation} with \cref{theo:suboptimality}, we see that with high probability, the suboptimality gap behaves as
$$
       \frac{J(\mathbf{G}_\star, {\mathbf{K}})^2 - J(\mathbf{G}_\star, {\mathbf{K}}_\star)^2}{J(\mathbf{G}_\star, {\mathbf{K}}_\star)^2} \sim \mathcal{O}\left(\frac{1}{\sqrt{N}}\right)\,.
$$
Despite the additional difficulty of hidden states, our sample complexity result is on the same level as that obtained in~\cite{dean2019sample} where a robust SLP procedure is used to design a robust LQR controller with full observations. 

\section{Conclusion and future work} \label{section:Conclusions}

We have developed a robust controller  synthesis procedure for partially observed LQG problems, by combining non-asymptotic identification methods with IOP for robust control. Our procedure is consistent with the idea of Coarse-ID control in~\cite{dean2019sample}, and extends the results in~\cite{dean2019sample} from fully observed state feedback systems to partially observed output-feedback systems that are open-loop stable. 

One interesting future direction is to extend these results to open-loop unstable systems. We note that non-asymptotic identification for partially observed open-loop unstable system is challenging in itself; see~\cite{zheng2020non} for a recent discussion. It is also non-trivial to derive a robust synthesis procedure with guaranteed performance, and using a pre-stabilizing controller might be useful~\citep{simchowitz2020improper,zheng2019systemlevel}. Finally, extending our results to an online adaptive setting and performing regret analysis are exciting future directions as well.

\bibliography{IEEEabrv,references2}

\preprintswitch{}{

\newpage

\appendix

\vspace{12mm}
\noindent \textbf{\Large Appendix}
\vspace{6mm}

This appendix is divided into seven parts. 

\begin{enumerate}
\setlength\itemsep{0em}
\item \cref{app:related_work} presents an extended literature review on non-asymptotic system identification, controller parameterization, as well as recent methods for online adaptive control from the RL community. 

    \item \cref{app:inequalities} presents some useful preliminaries on norm identities and inequalities~\cite[Chapter 4]{zhou1996robust}, and we also introduce the notion of internal stability.
    
\item \cref{app:robuststability} gives an overview of the IOP framework~\citep{furieri2019input} and proves the equivalent IOP formulation for robust LQG~\eqref{eq:robustLQG} in \cref{pr:robust_LQG}.

\item \cref{app:quasi_convexity} completes the proofs for establishing the quasi-convex approximation in the main text, including \cref{prop:LQGcostupperbound_step1}, \cref{prop:LQGcostupperbound} and \cref{theo:mainRobust}. 

\item \cref{app:theo:suboptimality} presents the proof for the suboptimality guarantee in \cref{theo:suboptimality}.

\item \cref{app:theo_FIR_identification} introduces some details of the OLS algorithm for system identification in~\cite{oymak2019non}, and we also present the proof for $\mathcal{H}_\infty$ bound in \cref{co:FIRestimation}.

\item \cref{app:h2norm} shows how to recast the infinite horizon LQG problem~\eqref{eq:LQG} as an equivalent $\mathcal{H}_2$ optimal control in terms of the system responses, defined in~\eqref{eq:cost_definition_convex_v2}. 
\end{enumerate}

Before proceeding with the technical proofs, we clarify some standard notation used in the paper for completeness. 
 The symbols 
 $\mathbb{R}$ and $\mathbb{N}$ refer to the set of real and integer numbers, respectively. 
We use lower and upper case letters (\emph{e.g.} $x$ and $A$) to denote vectors and matrices, respectively. Lower and upper case boldface letters (\emph{e.g.} $\mathbf{x}$ and $\mathbf{G}$) are used to denote signals and transfer matrices, respectively. 
Given a matrix $A \in \mathbb{R}^{m \times n}$, $A^\tr$ denotes the transpose of matrix $A$, and the Frobenius norm is denoted by $\|A\|_{\text{F}} = \sqrt{\text{Trace}(AA^\tr)}$. We denote $\|A\|$ as its spectral norm, \emph{i.e.}, its largest singular value $\sigma_{\max}(A)$. $\rho(A)$ denotes the spectral radius of a square matrix $A$. Multivariate normal distribution with mean $\mu$ and covariance matrix $\Sigma$ is denoted by $\mathcal{N}(\mu,\Sigma)$. 
We denote the set of real-rational proper stable transfer matrices as $\mathcal{RH}_{\infty}$.  
Given $\mathbf{G} \in \mathcal{RH}_{\infty}$, we denote its $\mathcal{H}_{\infty}$ norm by $\|\mathbf{G}\|_{\infty}:= \sup_{\omega} {\sigma}_{\max} (\mathbf{G}(e^{j\omega}))$, and the square of its $\mathcal{H}_2$ norm is $\|\mathbf{G}\|^2_{\mathcal{H}_2}:= {\frac{1}{2\pi}\int_{-\pi}^{\pi} \text{Trace}\left(\mathbf{G}^*(e^{j\omega})\mathbf{G}(e^{j\omega})\right)d\omega}$ which is finite for any stable $\mathbf{G}$ in the discrete time. For simplicity, we omit the dimension of transfer matrices, which shall be clear in the context. Also, we use $I$ (resp. $0$) to denote the identity matrix (resp. zero matrix) of compatible dimension.

\newpage

\section{Related work} \label{app:related_work}

\vspace{2mm}

\noindent \textbf{Estimation of linear time-invariant (LTI) systems:} Estimating system models from input/output data has a long history dating back to the sixties, especially in the case of LTI systems; see the excellent textbooks~\citep{ljung1999system,van2012subspace} and surveys~\citep{ljung2010perspectives,pintelon1994parametric, aastrom1971system}. While classical identification results~\citep{ljung1999system} focus on \emph{asymptotic} consistency of the proposed methods, our interest is primarily in contemporary \emph{non-asymptotic} results that rely heavily on concentration analysis of random matrices~\citep{tropp2015introduction}. Such finite-time analysis results provide both system model estimations and probabilistic uncertainty qualifications, naturally allowing for end-to-end performance guarantees of robust optimal control. We refer interest readers to~\cite{campi2002finite,vidyasagar2006learning} for early results and~\cite{chen2000control,parrilo1998mixed} for frequency-domain identification. 

With full state observations, it is shown that the ordinary least-squares (OLS) estimator  of system matrices achieves optimal non-asymptotic rates using either multiple independent trajectories~\citep{dean2019sample} or a single trajectory~\citep{simchowitz2018learning,sarkar2019near}. Interestingly,  with state observations, linear systems with a bigger spectral radius are easier to estimate using OLS~\citep{dean2019sample,sarkar2019near,simchowitz2018learning} (note that~\cite{sarkar2019near} points out a regularity condition for the consistency of OLS estimators when using a single trajectory). Recent work has also formulated different OLS procedures to estimate partially observed stable LTI systems using a single trajectory~\citep{oymak2019non,sarkar2019finite}, but their non-asymptotic rates degenerate as $\rho$, the spectral radius of the system matrix, approaches to one. A pre-filtering procedure was introduced in~\cite{simchowitz2019learning} that makes the performance of the OLS~\citep{oymak2019non} independent of $\rho \leq 1$, but this method still fails for open-loop unstable LTI systems. One way for estimating open-loop unstable partially observed systems is to use multiple independent trajectories; see~\cite{tu2017non,zheng2020non} for details.  
%


\vspace{2mm}
\noindent\textbf{Controller parameterization:} After obtaining the system dynamics, it is well-known that the problem of designing a controller is still non-convex. 
It might not be computationally efficient to search for the controller directly. We note that recent work has shown convergence guarantees for gradient-based algorithms over the controller parameter space despite the non-convexity issue~\citep{fazel2018global, malik2018derivative, furieri2019learning}. A celebrated class of techniques are based on solving Riccati equations that characterize the first-order necessary conditions of the optimal LQR/LQG controller~\citep{zhou1996robust}. 

Another class of classical techniques is based on \emph{controller parameterization}: to properly parameterize the set of stabilizing controllers using a change of variables in a different coordinate~\citep{boyd1994linear,gahinet1994linear, scherer1997multiobjective, zhou1996robust, youla1976modern}. Using appropriate controller parameterizations, the controller synthesis problem becomes convex in terms of the new variables, for which efficient algorithms exist to obtain the globally optimal solution. Lyapunov-based parameterizations are useful to derive linear matrix inequalities (LMIs) in the state-space domain~\citep{boyd1994linear,gahinet1994linear,scherer1997multiobjective}, and the classical Youla parameterization provides an elegant way to get a convex formulation in the frequency domain~\citep{youla1976modern,zhou1996robust}. Another two  recent frequency-domain parameterizations are the system-level parameterization (SLP)~\citep{wang2019system} and input-output parameterization (IOP)~\citep{furieri2019input}, which are equivalent to Youla; see~\citep{zheng2019equivalence,tseng2020realization} for a detailed comparison. The Youla, SLP, and IOP all treat certain closed-loop responses as design parameters, shifting from designing a controller to designing the closed loop responses directly, which allows for convex reformulation. This notion is known as \emph{closed-loop convexity}, as coined in~\citep{boyd1991linear}. A recent revisit of closed-loop convexity has identified all possible controller parameterizations using closed-loop responses, including SLP and IOP as special cases~\citep{zheng2019systemlevel}; see~\cite{tseng2020realization} for related discussions. 

\vspace{2mm}

\noindent\textbf{Offline robust control:} Given a fixed amount of data, non-asymptotic estimation provides a family of system models described by a nominal one with a set of bounded model errors. It is therefore necessary to design a controller that robustly stabilizes all the admissible plants. This problem, known as robust controller design, also has a rich history in control~\citep{zhou1996robust}. When the model errors are unstructured and allowed to be arbitrary norm-bounded LTI operators, traditional small-gain theorems and $\mathcal{H}_\infty$-type constraints can be used to solve the robust stabilization problem with no conservatism~\cite[Chapter 9]{zhou1996robust}. For problems with structured model errors, more sophisticated techniques such as $\mu$-synthesis techniques~\citep{doyle1982analysis}, integral quadratic constraints (IQC)~\citep{megretski1997system}, and sum-of-squares (SOS) optimization~\citep{scherer2006lmi} are known to be less conservative than small-gain approaches. 

Unlike robust stabilization, the results on quantifying the performance degradation of a robust controller due to model errors are far less complete. Building on a parameterization of robustly stabilizing controllers using the SLP~\citep{anderson2019system}, \cite{dean2019sample} developed a transparent method to quantify the impact of model uncertainty for the LQR problem. Their results show that the performance degradation of the robust controller scales \emph{linearly} with the model error~\citep{dean2019sample}. The SLP framework was also employed to deal with the output feedback control of stable single-input single-output (SISO) systems~\citep{boczar2018finite}. 
It is shown in~\cite{mania2019certainty} that the certainty equivalence principle achieves better closed-loop performance for both LQR and LQG problems, where the performance degradation scales as the \emph{square} of the estimation error, demonstrating an interesting trade-off between optimality and robustness. In parallel to the SLP, we develop a novel parameterization of robustly stabilizing controllers in the IOP framework~\citep{furieri2019input}. 
We further introduce a novel and tractable design methodology, under which the performance degradation of the robust controller for the LQG problem can be quantified. Similar to the SLP for LQR problems, our methodology based on the IOP allows robust control methods for LQG to interact, in a transparent way, with contemporary non-asymptotic estimation results, such as~\cite{tu2017non,oymak2019non,zheng2020non}.   

\vspace{2mm}
\noindent\textbf{Online adaptive control:} Data-driven learning and adaptation in controller design is not a new topic and goes back to the pioneering work in self-tuning regulators~\citep{kalman1958design,aastrom1973self}, which was followed by rich contributions in adaptive control theory~\citep{aastrom2013adaptive}. Here, we restrict our discussion to recent results on LQR/LQG problems; see~\citep{chen1986convergence,guo1996self} for early results. Much of recent work has focused on online LQR setting and provided regret bounds for LQR control of unknown systems~\citep{abbasi2011regret,dean2018regret}.~\cite{abbasi2011regret} proposes to use the optimism in the face of~uncertainty (OFU) principle combining with the classical Riccati-based solution, while~\cite{dean2018regret} adopts the SLP for controller parameterization and extends the robust synthesis procedure~\citep{dean2019sample} to online LQR setting.  For LQG problems, the first sub-linear regret bounds have been derived in \citep{simchowitz2020improper}; also see~\citep{lale2020adaptive,lale2020regret} for related discussions. The bounds have recently been improved in \citep{lale2020logarithmic} to logarithmic regret. Specifically, the Youla parameterization has been adopted in~\cite{simchowitz2020improper}, where the authors show that a simple online gradient descent algorithm over the Youla parameter achieves sub-linear regret in various settings. Policy gradient methods were also studied in~\cite{fazel2018global,mohammadi2019convergence} for the LQR problem and in \cite{furieri2019learning} for the LQG in a finite-horizon setting, where convergence to the globally optimal solution 
is guaranteed despite the non-convexity of the controller synthesis problem.

\section{Useful norm identities/inequalities and internal stability} \label{app:inequalities}
 
\subsection{Norm identities and inequalities }
For completeness, we collect some useful norm identities and inequalities that are key to our proofs. These norm identities and inequalities hold in both continuous-time and discrete-time settings, and they are standard and widely-used in the control literature. We refer the interested reader to~\cite[Chapter 4]{zhou1996robust} for details. 
 
 First, we have the following norm triangular inequalities
 \begin{equation}\label{eq:tr_ineq}
 \begin{aligned}
     &\norm{\mathbf{G}_1+\mathbf{G}_2}_{\mathcal{H}_2} \leq \norm{\mathbf{G}_1}_{\mathcal{H}_2}+\norm{\mathbf{G}_2}_{\mathcal{H}_2}\,, && \forall \mathbf{G}_1, \mathbf{G}_2 \in \mathcal{RH}_2\\
     & \norm{\mathbf{G}_1+\mathbf{G}_2}_{\infty} \leq \norm{\mathbf{G}_1}_{\infty}+\norm{\mathbf{G}_2}_{\infty}\,, && \forall \mathbf{G}_1, \mathbf{G}_2 \in \mathcal{RH}_\infty. 
      \end{aligned}
 \end{equation}
Second, we have separability for $\mathcal{H}_2$ norm as follows
\begin{equation} 
    \norm{\begin{bmatrix}\mathbf{G}_1&\mathbf{G}_2\\ \mathbf{G}_3&\mathbf{G}_4\end{bmatrix}}_{\mathcal{H}_2}^2={\norm{\mathbf{G}_1}^2_{\mathcal{H}_2}+\norm{\mathbf{G}_2}^2_{\mathcal{H}_2}+\norm{\mathbf{G}_3}^2_{\mathcal{H}_2}+\norm{\mathbf{G}_4}^2_{\mathcal{H}_2}} \label{eq:H2norm_separability}\,, \\
\end{equation}
for any $\mathbf{G}_1,\mathbf{G}_2,\mathbf{G}_3,\mathbf{G}_4 \in \mathcal{RH}_2$ with compatible dimensions. 
Since $\|\cdot\|_{{\infty}}$ is an induced norm, we have the following sub-multiplicative property 
\begin{equation} \label{eq:Hinf_ineq}
    \|\mathbf{G}_1\mathbf{G}_2\|_{{\infty}} \leq     \|\mathbf{G}_1\|_{{\infty}}\|\mathbf{G}_2\|_{\infty}, \qquad \forall \mathbf{G}_1, \mathbf{G}_2 \in \mathcal{RH}_\infty. 
\end{equation}
The inequality above does not hold for the $\mathcal{H}_{2}$  norm since it is not an induced norm. Instead, we have the following inequality, which generalizes $\|AB\|_{\text{F}} \leq \|A\|\|B\|_{\text{F}}$ and $\|AB\|_{\text{F}} \leq \|B\|\|A\|_{\text{F}}$. 
\begin{myLemma} \label{lemma:HinfH2inequality}
    Consider $\mathbf{G}_1 \in \mathcal{RH}_{\infty}$ and $\mathbf{G}_2 \in \mathcal{RH}_{2}$ with compatible dimension. We have 
    \begin{equation} 
        \|{\mathbf{G}_1\mathbf{G}_2}\|_{\mathcal{H}_2} \leq \|{\mathbf{G}_1}\|_{\infty}\|{\mathbf{G}_2}\|_{\mathcal{H}_2}, \qquad  \|{\mathbf{G}_1\mathbf{G}_2}\|_{\mathcal{H}_2} \leq \|{\mathbf{G}_1}\|_{\mathcal{H}_2}\|{\mathbf{G}_2}\|_{\infty}.\label{eq:H2Hinf_ineq}
    \end{equation}
\end{myLemma}
\begin{proof}
It is easy to verify the following property 
$$
    \begin{aligned}
        \|\mathbf{G}_1(z)\mathbf{G}_2(z)\|_{\mathcal{H}_2}^2 &= \frac{1}{2\pi} \int_{-\pi}^{\pi} \text{Trace}\left((\mathbf{G}_1(e^{j\omega})\mathbf{G}_2(e^{j\omega}))^*\mathbf{G}_1(e^{j\omega})\mathbf{G}_2(e^{j\omega})\right)d \omega \\
        &= \frac{1}{2\pi} \int_{-\pi}^{\pi} \text{Trace}\left(\mathbf{G}_2^*(e^{j\omega})\mathbf{G}_1^*(e^{j\omega})\mathbf{G}_1(e^{j\omega})\mathbf{G}_2(e^{j\omega})\right)d \omega \\
        &\leq \frac{1}{2\pi} \int_{-\pi}^{\pi} \text{Trace}\left(\mathbf{G}_2^*(e^{j\omega})\; \lambda_{\max}\left(\mathbf{G}_1^*(e^{j\omega})\mathbf{G}_1(e^{j\omega})\right) \;\mathbf{G}_2(e^{j\omega})\right)d \omega \\
        &\leq \sup_{\omega} \lambda_{\max}\left(\mathbf{G}_1^*(e^{j\omega})\mathbf{G}_1(e^{j\omega})\right) \times \frac{1}{2\pi} \int_{-\pi}^{\pi} \text{Trace}\left(\mathbf{G}_2^*(e^{j\omega})\mathbf{G}_2(e^{j\omega})\right)d \omega \\
        &= \|\mathbf{G}_1\|_{\infty}^2 \cdot \|\mathbf{G}_2\|_{\mathcal{H}_2}^2,
    \end{aligned}
$$
where the first inequality the following fact: for any matrices $A, B$ with compatible dimensions, we have
$$
    \text{Trace}(A^*B^*BA) \leq \lambda_{\max}(B^*B)\text{Trace}(A^*A).
$$
The other inequality can be proved similarly. This completes the proof. 
\end{proof}

Finally, we have the following useful inequalities from the celebrated \emph{small-gain theorem}~\cite[Theorem 9.1]{zhou1996robust}.  
\begin{myLemma} \label{lemma:smallgain_ineq}
If $\norm{\mathbf{G}_1}_{\infty}\times \norm{\mathbf{G}_2}_{\infty}<1 $, then
\begin{equation}
\label{eq:smallgain_ineq}
\begin{aligned}
\norm{(I - \mathbf{G}_1\mathbf{G}_2)^{-1}}_{\infty}\leq \frac{1}{1-\norm{\mathbf{G}_1}_{\infty}\norm{\mathbf{G}_2}_{\infty}} \,,\\
\norm{(I + \mathbf{G}_1\mathbf{G}_2)^{-1}}_{\infty}\leq \frac{1}{1-\norm{\mathbf{G}_1}_{\infty}\norm{\mathbf{G}_2}_{\infty}}.
\end{aligned}
\end{equation}
\end{myLemma}
This lemma is standard and can be proved by taking a power expansion of the matrix inverse and using \eqref{eq:tr_ineq},~\eqref{eq:Hinf_ineq} and the fact that 
$$
\sum_{k=0}^\infty \norm{(-1)^{k}\mathbf{M}^k}_{\infty} = \sum_{k=0}^\infty \norm{\mathbf{M}^k}_{\infty} \leq   \sum_{k=0}^\infty \norm{\mathbf{M}}_{\infty}^k = \frac{1}{1-\norm{\mathbf{M}}_{\infty}}.
$$

To get the exact $\mathcal{H}_{\infty}$ norm of an FIR transfer matrix, the following result, from~\cite[Theorem 5.8]{dumitrescu2007positive}, offers a computational method based on semidefinite programming.

\begin{myLemma} \label{lemma:Hinfnorm}
    Let $\mathbf{U}(z)$ be the following FIR system
\begin{equation} \label{eq:FIR}
     \mathbf{U} = \sum_{t=0}^T U_t\frac{1}{z^t}, \qquad \text{where} \quad  U_t \in \mathbb{R}^{p \times m}.
\end{equation}
The relation $\|\mathbf{U}\|_{\infty} \leq \gamma
$ holds if and only if there exists a positive semidefinite matrix $Q \in \mathbb{S}^{p(T+1)}_+$ such that
$$
    \begin{bmatrix}
        Q & \hat{U}\\
        \hat{U}^\tr & \gamma I_m
    \end{bmatrix} \succeq 0, \quad  \sum_{j=1}^{T-k} Q_{j+k,j} = \gamma \delta_kI_p, \; k = 0, 1, 2, \ldots, T, 
$$
where $Q_{j,k} \in \mathbb{R}^{p \times p}$ denotes the $(j,k)$-th block of $Q$,  and
$$
    \hat{U} = \begin{bmatrix}
    U_0 \\
    U_1 \\
    \vdots \\
    U_T
    \end{bmatrix} \in \mathbb{R}^{p(T+1) \times m}, \quad  \delta_k = \begin{cases} 1, & k = 0\\
    0, & \text{otherwise.}\end{cases}
$$
\end{myLemma}

\subsection{Internally stabilizing controllers}
\label{app:internally_stabilizing}

Applying the dynamic controller~\eqref{eq:dyController} to the LTI system~\eqref{eq:dynamic}, we have the closed-loop system as 
$$
    \begin{bmatrix}x_{t+1} \\ \xi_{t+1} \end{bmatrix} = \begin{bmatrix}
    A_\star + B_\star D_k C_\star & B_\star C_k\\
  B_kC_\star & A_k
    \end{bmatrix}\begin{bmatrix}x_{t} \\ \xi_{t} \end{bmatrix} +  \begin{bmatrix}
    B_\star & 0 \\
    0 & B_k
    \end{bmatrix} \begin{bmatrix}
    w_t \\
    v_t
    \end{bmatrix}.
$$
The closed-loop system is called \emph{internally stable} if and only if~\citep{zhou1996robust} 
$$
    \rho\left(\begin{bmatrix}
    A_\star + B_\star D_k C_\star & B_\star C_k\\
  B_kC_\star & A_k
    \end{bmatrix}\right) < 1. 
$$

We say the controller $\mathbf{K}$ \emph{internally stabilizes} the plant $\mathbf{G}_\star$ if the closed-loop system is {internally stable}. The set of all LTI internally stabilizing controllers is defined as
\begin{equation} \label{eq:internallystabilizing}
    \mathcal{C}_{\text{stab}} := \{\mathbf{K} \mid \mathbf{K} \; \text{internally stabilizes} \; \mathbf{G}_\star\}.
\end{equation}
Note that $\mathcal{C}_{\text{stab}} \neq \emptyset$ if and only if \cref{assumption:stabilizability} holds~\citep{zhou1996robust}. \cref{assumption:open_loop_stability} means that $\mathbf{K} = 0 \in  \mathcal{C}_{\text{stab}}$. The LQG formulation~\eqref{eq:LQG} has an implicit constraint $\mathbf{K} \in \mathcal{C}_{\text{stab}}$; otherwise the cost function goes to infinity. It is well-known that $\mathcal{C}_{\text{stab}}$ is non-convex, posing a difficulty searching for $\mathbf{K}$ directly. It is not difficult to find explicit examples where $\mathbf{K}_1, \mathbf{K}_2\in \mathcal{C}_{\text{stab}}$ and $\frac{1}{2}(\mathbf{K}_1 + \mathbf{K}_2)\notin \mathcal{C}_{\text{stab}}$. 
A state-space characterization of  $\mathcal{C}_{\text{stab}}$ is 
\begin{equation*} 
    \mathcal{C}_{\text{stab}} = \left\{\mathbf{K} = C_k(zI - A_k)^{-1}B_k + D_k \mid   \rho\left(\begin{bmatrix}
    A_\star + B_\star D_k C_\star & B_\star C_k\\
  B_kC_\star & A_k
    \end{bmatrix}\right) < 1.  \right\}.
\end{equation*}
Again, the spectral radius condition is non-convex in terms of parameters $A_k,B_k,C_k,D_k$.   
Suitable controller parameterizations are vital in many synthesis procedures~\citep{zhou1996robust, boyd1991linear, francis1987course}, such as the classical Youla parameterization~\citep{youla1976modern}, the recent SLP~\citep{wang2019system} and IOP~\citep{furieri2019input}. All these three parameterizations are based closed-loop responses in the frequency domain; see~\citep{zheng2019equivalence,zheng2019systemlevel} for a comparison.

\section{The IOP framework for robust LQG~\eqref{eq:robustLQG} and Proof of \cref{pr:robust_LQG}} \label{app:robuststability}

\subsection{Useful results from input-output parameterization (IOP)}

Similar to the Youla parameterization~\citep{youla1976modern} and the SLP~\citep{wang2019system}, the IOP framework focuses on the \emph{system responses} of a closed-loop system. In particular, given an arbitrary control policy $\mathbf{u} = \mathbf{K} \mathbf{y}$, straightforward calculations show that the closed-loop responses from the noises $\mathbf{v}, \mathbf{w}$ to the output $\mathbf{y}$ and control action $\mathbf{u}$ are 
\begin{equation} \label{eq:responses}
    \begin{bmatrix} \mathbf{y} \\ \mathbf{u} \end{bmatrix} = \begin{bmatrix}(I-\mathbf{G}_\star\mathbf{K})^{-1}&(I-\mathbf{G}_\star\mathbf{K})^{-1}\mathbf{G}_\star\\\mathbf{K}(I-\mathbf{G}_\star\mathbf{K})^{-1}&(I-\mathbf{K}\mathbf{G}_\star)^{-1}\end{bmatrix} \begin{bmatrix} \mathbf{v} \\ \mathbf{w} \end{bmatrix}.
\end{equation}
To ease the notation, we can define the closed-loop responses as 
\begin{equation} \label{eq:YUWZ}
     \begin{bmatrix} \mathbf{Y} & \mathbf{W}\\\mathbf{U}&\mathbf{Z}\end{bmatrix} := \begin{bmatrix}(I-\mathbf{G}_\star\mathbf{K})^{-1}&(I-\mathbf{G}_\star\mathbf{K})^{-1}\mathbf{G}_\star\\\mathbf{K}(I-\mathbf{G}_\star\mathbf{K})^{-1}&(I-\mathbf{K}\mathbf{G}_\star)^{-1}\end{bmatrix}.
\end{equation}
Then, we have the following theorem that gives an parameterization of all internally stabilizing controllers.

\begin{myTheorem}[{Input-output parameterization~\cite{furieri2019input}}] \label{theo:iop}
Consider the LTI system~\eqref{eq:dynamic}, evolving under a dynamic control policy~\eqref{eq:dyController}. The following statements are true:
\begin{enumerate}
    \item For any internally stabilizing controller $\mathbf{K}$, the resulting closed-loop responses~\eqref{eq:YUWZ} satisfy 
    \begin{subequations} \label{eq:affineYUWZ}
        \begin{align}
        \begin{bmatrix}
        	 I&-\mathbf{G}_\star
        	 \end{bmatrix}\begin{bmatrix}
        	 \mathbf{Y}&\mathbf{W}\\\mathbf{U}&\mathbf{Z}
        	 \end{bmatrix}&=\begin{bmatrix}
        	 I&0
        	 \end{bmatrix},  \label{eq:affineYUWZ_s1}\\
        	 \begin{bmatrix}
        	 \mathbf{Y}&\mathbf{W}\\\mathbf{U}&\mathbf{Z}
        	 \end{bmatrix}
        	 \begin{bmatrix}
        	 -\mathbf{G}_\star\\I
        	 \end{bmatrix}&=\begin{bmatrix}
        	 0\\I
        	 \end{bmatrix}, \label{eq:affineYUWZ_s2}\\
        	 \mathbf{Y}, \mathbf{U}, \mathbf{W}, \mathbf{Z} &\in \mathcal{RH}_\infty. \label{eq:affineYUWZ_s3}
        \end{align}
    \end{subequations}
\item For any transfer matrices $\mathbf{Y}, \mathbf{U}, \mathbf{W}, \mathbf{Z}$ satisfying~\eqref{eq:affineYUWZ_s1}---\eqref{eq:affineYUWZ_s3}, the controller $\mathbf{K} = \mathbf{U}\mathbf{Y}^{-1}$ internally stabilizes the plant $\mathbf{G}_{\star}$. 
\end{enumerate}
\end{myTheorem}
\cref{theo:iop} presents a set of sufficient and necessary conditions on closed-loop responses that are achievable by an internally stabilizing controller $\mathbf{K}$. This theorem holds for both open-loop stable and unstable plants. Note that~\eqref{eq:affineYUWZ} defines a set of convex constraints on the closed-loop responses, which gives a convex parameterization of the set of stabilizing controller $\mathcal{C}_{\text{stab}}$~\eqref{eq:internallystabilizing} in the frequency domain.  This equivalence in \cref{theo:iop} allows us to directly work with the closed-loop responses, leading to convex reformulations for a variety of controller synthesis problems~\citep{furieri2019input,zheng2019equivalence,zheng2019systemlevel}.

From \cref{theo:iop} and the standard equivalence between infinite horizon LQG and $\mathcal{H}_2$ optimal control,  the LQG problem~\eqref{eq:LQG} can be equivalently written as 
\begin{equation} \label{eq:LQG_YUWZ}
    \begin{aligned}
        \min_{\mathbf{Y},\mathbf{U},\mathbf{W},\mathbf{Z}} \quad & \left\|\begin{bmatrix}Q^{\frac{1}{2}}&0\\0&R^{\frac{1}{2}}\end{bmatrix} \begin{bmatrix} \mathbf{Y} & \mathbf{W}\\\mathbf{U}&\mathbf{Z}\end{bmatrix} \begin{bmatrix} \sigma_vI_p & \\ & \sigma_wI_m \end{bmatrix} \right\|^2_{\mathcal{H}_2} \\
        \st \quad & ~\eqref{eq:affineYUWZ_s1}-\eqref{eq:affineYUWZ_s3},
    \end{aligned}
\end{equation}
and the controller is recovered by $\mathbf{K} = \mathbf{U}\mathbf{Y}^{-1}$. We provide a full derivation of this equivalence in \cref{app:h2norm}. For notational simplicity, but without loss of generality, we assume $Q=I_p, R=I_m, \sigma_v = 1, \sigma_w = 1$. Then, given any internally stabilizing $\mathbf{K}$, we can write $\mathbf{K}=\mathbf{UY}^{-1}$ and the square root of the LQG cost as 
\begin{equation}
\label{eq:cost_definition_convex}
J(\mathbf{G}_{\star},\mathbf{K})=\left\|\begin{bmatrix}\mathbf{Y}&\mathbf{W}\\ \mathbf{U}&\mathbf{Z}\end{bmatrix}\right\|_{\mathcal{H}_2},
\end{equation}
with the closed-loop responses $(\mathbf{Y},\mathbf{U},\mathbf{W},\mathbf{Z})$ defined in~\eqref{eq:YUWZ}.

\subsection{IOP for robustly stabilizing controllers}

We are now ready to derive the convex parameterization of robustly stabilizing controllers used in \cref{pr:robust_LQG}. 

    \begin{myTheorem} \label{theo:robuststabilization}
       Let $\mathcal{C}(\epsilon)$ define the set of robustly stabilizing controller as:
       \begin{equation*}
    \mathcal{C}(\epsilon) :=\left\{\mathbf{K} \mid \mathbf{K} \; \text{internally stabilizes} \; \mathbf{G}=\hat{\mathbf{G}} + \mathbf{\Delta}, \forall \|\mathbf{\Delta}\|_{\infty} < \epsilon \right\}.
\end{equation*}
       Then, the following statements are true.
        \begin{enumerate}
          \item For any robustly stabilizing controllers $\mathbf{K} \in \mathcal{C}(\epsilon)$, applying $\mathbf{K}$ to the nominal plant $\hat{\mathbf{G}}$ leads to a set of system responses $(\hat{\mathbf{Y}}, \hat{\mathbf{U}}, \hat{\mathbf{W}}, \hat{\mathbf{Z}})$ satisfying 
          \begin{subequations} \label{Eq:Param}
    	 \begin{align}
    	 &\begin{bmatrix} 	 I&-\hat{\mathbf{G}} \end{bmatrix}
    	 \begin{bmatrix}
    	 \hat{\mathbf{Y}} & \hat{\mathbf{W}} \\ \hat{\mathbf{U}} & \hat{\mathbf{Z}}
    	 \end{bmatrix}=\begin{bmatrix}   I & 0 \end{bmatrix}\,, \label{eq:aff1}\\
    	 & \begin{bmatrix} \hat{\mathbf{Y}} & \hat{\mathbf{W}} \\ \hat{\mathbf{U}} & \hat{\mathbf{Z}}
    	 \end{bmatrix}\begin{bmatrix}  	 \hat{\mathbf{G}}\\I
    	 \end{bmatrix}=\begin{bmatrix}
    	 0\\I
    	 \end{bmatrix}\label{eq:aff2}\,,\\
    	 &\begin{matrix}
    	\hat{\mathbf{Y}}, \hat{\mathbf{U}}, \hat{\mathbf{W}}, \hat{\mathbf{Z}} \in \mathcal{RH}_\infty,
    	 \end{matrix}\label{eq:aff3}\\
     & \|\hat{\mathbf{U}}\|_{\infty} \leq \frac{1}{\epsilon}. \label{eq:aff4}
    	 \end{align}
	 \end{subequations}
          \item For any transfer matrices $(\hat{\mathbf{Y}}, \hat{\mathbf{U}}, \hat{\mathbf{W}}, \hat{\mathbf{Z}})$ satisfying~\eqref{eq:aff1}-\eqref{eq:aff4}, the controller $\mathbf{K} = \hat{\mathbf{U}}\hat{\mathbf{Y}}^{-1} \in \mathcal{C}(\epsilon)$.
        \end{enumerate}
    \end{myTheorem}

\begin{proof}
    We prove the first statement.  Let $\mathbf{K} \in \mathcal{C}(\epsilon)$. Then, $\mathbf{K}$ internally stabilizes $\hat{\mathbf{G}}$. The closed-loop system responses are
      $$
        \begin{aligned}
            \hat{\mathbf{Y}} &= (I - \hat{\mathbf{G}}\mathbf{K})^{-1}, & \hat{\mathbf{W}} &= (I - \hat{\mathbf{G}}\mathbf{K})^{-1}\hat{\mathbf{G}}, \\
           \hat{\mathbf{U}} &= \mathbf{K}(I - \hat{\mathbf{G}}\mathbf{K})^{-1},  &            \hat{\mathbf{Z}} &= (I - \mathbf{K}\hat{\mathbf{G}})^{-1}.
        \end{aligned}
      $$
      Simple algebra verifies that $(\hat{\mathbf{Y}}, \hat{\mathbf{U}}, \hat{\mathbf{W}}, \hat{\mathbf{Z}})$ satisfy~\eqref{eq:aff1}-\eqref{eq:aff2}. By definition, they are all stable, thus~\eqref{eq:aff3} holds.  Since  $\mathbf{K} \in \mathcal{C}(\epsilon)$, we have
      $$
        \mathbf{K}(I - \mathbf{G}\mathbf{K})^{-1} \in \mathcal{RH}_{\infty}, \quad \text{where}\;\; \mathbf{G} = \hat{\mathbf{G}} + \mathbf{\Delta}, \forall \|\mathbf{\Delta}\|_\infty < \epsilon.
      $$
      Considering the fact
      $$
        \begin{aligned}
            \mathbf{K}(I - \mathbf{G}\mathbf{K})^{-1} &= \mathbf{K}(I - (\hat{\mathbf{G}} + \mathbf{\Delta})\mathbf{K})^{-1} \\
            &= \mathbf{K}(I - \hat{\mathbf{G}}\mathbf{K} - \mathbf{\Delta}\mathbf{K})^{-1} \\
            &= \mathbf{K}(I -  \hat{\mathbf{G}}\mathbf{K})^{-1}(I - \mathbf{\Delta}\mathbf{K}(I -  \hat{\mathbf{G}}\mathbf{K})^{-1})^{-1} \\
            &= \hat{\mathbf{U}}(I - \mathbf{\Delta}\hat{\mathbf{U}})^{-1} \in \mathcal{RH}_{\infty}, \qquad \forall \|\mathbf{\Delta}\|_{\infty} < \epsilon,
        \end{aligned}
      $$
      we have
      $
           \mathbf{\Delta}\mathbf{K}(I - \mathbf{G}\mathbf{K})^{-1} =  \mathbf{\Delta}\hat{\mathbf{U}}(I - \mathbf{\Delta}\hat{\mathbf{U}})^{-1}  = I - (I - \mathbf{\Delta}\hat{\mathbf{U}})^{-1}  \in \mathcal{RH}_{\infty},\; \forall \|\mathbf{\Delta}\|_{\infty} < \epsilon,
      $
      which implies  
      $$
        (I - \mathbf{\Delta}\hat{\mathbf{U}})^{-1}  \in \mathcal{RH}_{\infty}, \; \forall \|\mathbf{\Delta}\|_{\infty} < \epsilon.
      $$ 
      By the small gain theorem~\cite[Theorem 9.1]{zhou1996robust}, we have~\eqref{eq:aff4} holds.
      
     \textit{Proof of Statement 2:}  Consider $(\hat{\mathbf{Y}}, \hat{\mathbf{U}}, \hat{\mathbf{W}}, \hat{\mathbf{Z}})$ satisfying~\eqref{eq:aff1}-\eqref{eq:aff4}, and let the controller $\mathbf{K} = \hat{\mathbf{U}}\hat{\mathbf{Y}}^{-1}$. To show $\mathbf{K} = \hat{\mathbf{U}}\hat{\mathbf{Y}}^{-1} \in \mathcal{C}(\epsilon)$, we only need to prove that
        \begin{equation}
        \begin{aligned}
            (I - (\hat{\mathbf{G}}+\mathbf{\Delta})\mathbf{K})^{-1} &\in \mathcal{RH}_{\infty}, \\
            \mathbf{K}(I - (\hat{\mathbf{G}}+\mathbf{\Delta})\mathbf{K})^{-1} &\in \mathcal{RH}_{\infty}, \\
            (I - (\hat{\mathbf{G}}+\mathbf{\Delta})\mathbf{K})^{-1}(\hat{\mathbf{G}}+\mathbf{\Delta}) &\in \mathcal{RH}_{\infty}, \\
            (I - \mathbf{K}(\hat{\mathbf{G}}+\mathbf{\Delta}))^{-1} &\in \mathcal{RH}_{\infty}, 
        \end{aligned} \qquad \forall \|\mathbf{\Delta}\|_{\infty} < \epsilon.  
        \end{equation}
      Since $\|\mathbf{\Delta}\|_{\infty} < \epsilon, \|\hat{\mathbf{U}}\|_{\infty} \leq \frac{1}{\epsilon}$, by the small gain theorem, we have
      $$
        (I - \mathbf{\Delta}\hat{\mathbf{U}})^{-1} \in \mathcal{RH}_{\infty}, \qquad (I - \hat{\mathbf{U}}\mathbf{\Delta})^{-1} \in \mathcal{RH}_{\infty}.
      $$
      Now, some algebra show that
      \begin{equation*} 
             \begin{aligned}
            (I - (\hat{\mathbf{G}}+\mathbf{\Delta})\mathbf{K})^{-1} &=
            (I - (\hat{\mathbf{G}}+\mathbf{\Delta})\hat{\mathbf{U}}\mathbf{Y}^{-1})^{-1}\\
            &=\hat{\mathbf{Y}}(\hat{\mathbf{Y}} - (\hat{\mathbf{G}}+\mathbf{\Delta})\hat{\mathbf{U}})^{-1} \\
            &=\hat{\mathbf{Y}}(I - \mathbf{\Delta}\hat{\mathbf{U}})^{-1} \in \mathcal{RH}_{\infty},
            \end{aligned}
          \end{equation*}
    and that
         \begin{equation*} 
             \begin{aligned}
            \mathbf{K}(I - (\hat{\mathbf{G}}+\mathbf{\Delta})\mathbf{K})^{-1} = \mathbf{U}\mathbf{Y}^{-1}\mathbf{Y}(I - \mathbf{\Delta}\mathbf{U})^{-1} = \mathbf{U}(I - \mathbf{\Delta}\mathbf{U})^{-1} \in \mathcal{RH}_{\infty}.
            \end{aligned}
          \end{equation*}
      Similarly, we can show 
           $
            (I - \mathbf{K}(\hat{\mathbf{G}}+\mathbf{\Delta}))^{-1} = (I - \hat{\mathbf{U}}\mathbf{\Delta})^{-1}\hat{\mathbf{Z}} \in \mathcal{RH}_{\infty}.
          $
        Finally, we can verify that
         $$
             \begin{aligned}
            (I - (\hat{\mathbf{G}}+\mathbf{\Delta})\mathbf{K})^{-1}(\hat{\mathbf{G}}+\mathbf{\Delta}) &= (I - \hat{\mathbf{G}}\mathbf{K} - \mathbf{\Delta}\mathbf{K})^{-1}(\hat{\mathbf{G}}+\mathbf{\Delta}) \\
            &= (I - (I - \hat{\mathbf{G}}\mathbf{K})^{-1}\mathbf{\Delta}\mathbf{K})^{-1}(I - \hat{\mathbf{G}}\mathbf{K})^{-1}(\hat{\mathbf{G}}+\mathbf{\Delta}) \\
            &= (I - \hat{\mathbf{Y}}\mathbf{\Delta}\mathbf{K})^{-1}\hat{\mathbf{Y}}(\hat{\mathbf{G}}+\mathbf{\Delta}) \\
            &= (I - \hat{\mathbf{Y}}\mathbf{\Delta}\mathbf{K})^{-1}\hat{\mathbf{W}} + (I - \hat{\mathbf{Y}}\mathbf{\Delta}\mathbf{K})^{-1}\mathbf{\Delta} \in \mathcal{RH}_\infty,
            \end{aligned}
          $$
         where the last statement applies the fact that 
          $$
             \text{det}(I - \hat{\mathbf{Y}}\mathbf{\Delta}\mathbf{K}) = \text{det}(I - \mathbf{\Delta}\mathbf{K}\hat{\mathbf{Y}})  = \text{det}(I - \mathbf{\Delta}\hat{\mathbf{U}}) \neq 0,
          $$
          for all  $z$ within the unit circle, indicating $(I - \hat{\mathbf{Y}}\mathbf{\Delta}\mathbf{K})^{-1}$ has no poles within the unite circle, \emph{i.e.}  $(I - \hat{\mathbf{Y}}\mathbf{\Delta}\mathbf{K})^{-1} \in \mathcal{RH}_{\infty}$.  We complete the proof. 
\end{proof}

Note that~\eqref{eq:aff1}-\eqref{eq:aff4} defines a convex set in $\hat{\mathbf{Y}}, \hat{\mathbf{U}}, \hat{\mathbf{W}}, \hat{\mathbf{Z}}$. \cref{theo:robuststabilization} holds for both open-loop stable and unstable plants $\hat{\mathbf{G}}$. However, as discussed in~\cite{zheng2019systemlevel}, a notion of numerical robustness cannot be guaranteed for open-loop unstable plants. Finally, we note that using~\cite[Theroem 9.4 \& Theorem 9.5]{zhou1996robust}, it is possible to extend \cref{theo:robuststabilization} to cover perturbed models with additive or multiplicative uncertainty
$$
    \mathbf{\Pi}:= \{\mathbf{G} = \hat{\mathbf{G}} + \mathbf{W}_1\mathbf{\Delta}\mathbf{W}_2, \; \|\mathbf{\Delta}\|_{\infty} < \epsilon\}
\quad \text{or} \quad
    \mathbf{\Pi}:= \{\mathbf{G} = (I + \mathbf{W}_1\mathbf{\Delta}\mathbf{W}_2)\hat{\mathbf{G}}, \|\mathbf{\Delta}\|_{\infty} < \epsilon\},
$$
where $\mathbf{W}_1$ and $\mathbf{W}_2$ are pre-fixed stable transfer matrices. 

\subsection{Proof of \cref{pr:robust_LQG}}

From \cref{theo:robuststabilization}, the following lemma is immediate. 

\begin{myLemma} \label{lemma:robuststability}
    The set of robustly stabilizing controllers $ \mathcal{C}(\epsilon)$ can be convexly parameterized as 
    $$
         \mathcal{C}(\epsilon) = \left\{ \mathbf{K} = \hat{\mathbf{U}}\hat{\mathbf{Y}}^{-1} \left|   \begin{aligned}
        \begin{bmatrix}
        	 I&-\hat{\mathbf{G}}
        	 \end{bmatrix}\begin{bmatrix}
        	 \hat{\mathbf{Y}} & \hat{\mathbf{W}}\\\hat{\mathbf{U}} & \hat{\mathbf{Z}}
        	 \end{bmatrix}&=\begin{bmatrix}
        	 I&0
        	 \end{bmatrix}, \\
        	 \begin{bmatrix}
         \hat{\mathbf{Y}} & \hat{\mathbf{W}}\\\hat{\mathbf{U}} & \hat{\mathbf{Z}}
        	 \end{bmatrix}
        	 \begin{bmatrix}
        	 -\hat{\mathbf{G}}\\I
        	 \end{bmatrix}&=\begin{bmatrix}
        	 0\\I
        	 \end{bmatrix}, \\
        	  \hat{\mathbf{Y}}, \hat{\mathbf{W}}, \hat{\mathbf{U}},  \hat{\mathbf{Z}} \in \mathcal{RH}_\infty, \, &\|\hat{\mathbf{U}}\|_{\infty} \leq \frac{1}{\epsilon}. 
        \end{aligned} \right.\right\}.
    $$
\end{myLemma}

\begin{myLemma} \label{lemma:robustcost}
    Given $\mathbf{K} \in \mathcal{C}(\epsilon)$, we let $\hat{\mathbf{Y}}, \hat{\mathbf{W}}, \hat{\mathbf{U}},  \hat{\mathbf{Z}}$ be the corresponding system responses~\eqref{eq:YUWZ} on the estimated system $\hat{\mathbf{G}}$. When applying $\mathbf{K}$ to the true plant $\mathbf{G}_{\star}$, the {square root} of the LQG cost~\eqref{eq:cost_definition_convex} can be equivalently written as 
    \begin{equation} \label{eq:LQGcostrobust}
        J(\mathbf{G}_{\star},\mathbf{K}) = \left\|\begin{bmatrix}\hat{\mathbf{Y}}(I-\mathbf{\Delta}\hat{\mathbf{U}})^{-1}&\hat{\mathbf{Y}}(I-\mathbf{\Delta}\hat{\mathbf{U}})^{-1}(\hat{\mathbf{G}}+\mathbf{\Delta})\\\hat{\mathbf{U}}(I-\mathbf{\Delta}\hat{\mathbf{U}})^{-1}&(I-\hat{\mathbf{U}}\mathbf{\Delta})^{-1}\hat{\mathbf{Z}}\end{bmatrix}\right\|_{\mathcal{H}_2}. 
    \end{equation}
\end{myLemma}
\begin{proof}
    By definition, the controller $\mathbf{K} \in  \mathcal{C}(\epsilon)$ internally stabilizes $\mathbf{G}_{\star}$, and the cost~\eqref{eq:cost_definition_convex} is finite. Now, applying  $\mathbf{K}$ to the true plant $\mathbf{G}_{\star}$, we have the closed-loop responses in~\eqref{eq:responses}. Since $\hat{\mathbf{Y}}, \hat{\mathbf{W}}, \hat{\mathbf{U}},  \hat{\mathbf{Z}}$ are the corresponding system responses~\eqref{eq:YUWZ} on the estimated system $\hat{\mathbf{G}}$, we have
    $$
    \begin{aligned}
        (I-\mathbf{G}_\star\mathbf{K})^{-1} &= (I-(\hat{\mathbf{G}} + \mathbf{\Delta})\mathbf{K})^{-1} \\
        & = (I - \hat{\mathbf{G}}\mathbf{K})^{-1}(I -  \mathbf{\Delta}\mathbf{K}(I - \hat{\mathbf{G}}\mathbf{K})^{-1})^{-1} \\
        & = \hat{\mathbf{Y}}(I - \mathbf{\Delta}\hat{\mathbf{U}})^{-1}. 
    \end{aligned}
    $$
    Similarly, we can derive
    $$
        \begin{aligned}
            \mathbf{K}(I-\mathbf{G}_\star\mathbf{K})^{-1} &= \hat{\mathbf{U}}(I - \mathbf{\Delta}\hat{\mathbf{U}})^{-1} \\
            (I-\mathbf{G}_\star\mathbf{K})^{-1}\mathbf{G}_\star&= \hat{\mathbf{Y}}(I-\mathbf{\Delta}\hat{\mathbf{U}})^{-1}(\hat{\mathbf{G}}+\mathbf{\Delta}) \\
            (I - \mathbf{K}\mathbf{G}_\star)^{-1} &=  (I-\hat{\mathbf{U}}\mathbf{\Delta})^{-1}\hat{\mathbf{Z}}\,.
        \end{aligned}
    $$
    Thus, the result~\eqref{eq:LQGcostrobust} follows directly.
\end{proof}

Combining \cref{lemma:robuststability} and \cref{lemma:robustcost} allows us to recast the robust LQG problem~\eqref{eq:robustLQG} into an equivalent formulation in the frequency domain~\eqref{eq:robustLQG_YUWZ}. This completes the proof of \cref{pr:robust_LQG}.

\section{Quasi-convex formulation} \label{app:quasi_convexity}

\subsection{Proof of \cref{prop:LQGcostupperbound_step1}}
\label{app:upperbound_lastterm}
    Since $\|\mathbf{\Delta}\hat{\mathbf{U}}\|_{\infty} < 1$, the following power series expansion holds
$$
    (I - \mathbf{\Delta}\hat{\mathbf{U}})^{-1} = I + \sum_{k=1}^{\infty} (\mathbf{\Delta}\hat{\mathbf{U}})^k.
$$
Then, we have 
\begin{equation*}
    \begin{aligned}
        \hat{\mathbf{Y}}(I-\mathbf{\Delta}\hat{\mathbf{U}})^{-1}(\hat{\mathbf{G}}+\mathbf{\Delta}) &=  \hat{\mathbf{Y}}\left(I + \sum_{k=1}^{\infty} (\mathbf{\Delta}\hat{\mathbf{U}})^k\right)(\hat{\mathbf{G}}+\mathbf{\Delta}) \\
        &= \hat{\mathbf{W}} +  \hat{\mathbf{Y}}\mathbf{\Delta} + \hat{\mathbf{Y}}\left(\sum_{k=1}^\infty(\mathbf{\Delta}\hat{\mathbf{U}})^k\right) (\hat{\mathbf{G}}+\mathbf{\Delta}). 
    \end{aligned}
\end{equation*}
Now, considering the norm inequalities in~\eqref{eq:tr_ineq},~\eqref{eq:Hinf_ineq}, and \cref{lemma:HinfH2inequality} (\cref{app:inequalities}), we have 
\begin{equation*}
    \begin{aligned}
       &\quad \|\hat{\mathbf{Y}}(I-\mathbf{\Delta}\hat{\mathbf{U}})^{-1}(\hat{\mathbf{G}}+\mathbf{\Delta})\|_{\mathcal{H}_2} \\
       &\leq \|\hat{\mathbf{W}}\|_{\mathcal{H}_2} +  \|\hat{\mathbf{Y}}\mathbf{\Delta}\|_{\mathcal{H}_2} + \left\|\hat{\mathbf{Y}}\left(\sum_{k=1}^\infty(\mathbf{\Delta}\hat{\mathbf{U}})^k\right) (\hat{\mathbf{G}}+\mathbf{\Delta})\right\|_{\mathcal{H}_2} \\
       &\leq \|\hat{\mathbf{W}}\|_{\mathcal{H}_2} +  \epsilon \|\hat{\mathbf{Y}}\|_{\mathcal{H}_2} + \|\hat{\mathbf{Y}}\|_{\mathcal{H}_2}\left(\sum_{k=1}^\infty \epsilon^k \|\hat{\mathbf{U}}\|_{\infty}^k\right) (\|\hat{\mathbf{G}}\|_{\infty}+\epsilon) \\
       &= \|\hat{\mathbf{W}}\|_{\mathcal{H}_2} + \epsilon \|\hat{\mathbf{Y}}\|_{\mathcal{H}_2} + \|\hat{\mathbf{Y}}\|_{\mathcal{H}_2} \frac{\epsilon \|\hat{\mathbf{U}}\|_{\infty}(\|\hat{\mathbf{G}}\|_\infty + \epsilon)}{1 - \epsilon \|\hat{\mathbf{U}}\|_{\infty}}.
       \end{aligned}
\end{equation*}
Now, we finish the proof by observing
\begin{equation*} 
    \begin{aligned}
       &\quad \|\hat{\mathbf{Y}}(I-\mathbf{\Delta}\hat{\mathbf{U}})^{-1}(\hat{\mathbf{G}}+\mathbf{\Delta})\|_{\mathcal{H}_2} \\
       &\leq  \frac{\|\hat{\mathbf{W}}\|_{\mathcal{H}_2} + \epsilon \|\hat{\mathbf{Y}}\|_{\mathcal{H}_2} + \|\hat{\mathbf{Y}}\|_{\mathcal{H}_2}\epsilon \|\hat{\mathbf{U}}\|_{\infty}(\|\hat{\mathbf{G}}\|_\infty + \epsilon)}{1 - \epsilon \|\hat{\mathbf{U}}\|_{\infty}} \\
       &= \frac{\|\hat{\mathbf{W}}\|_{\mathcal{H}_2} + \epsilon \|\hat{\mathbf{Y}}\|_{\mathcal{H}_2} + \epsilon \|\hat{\mathbf{U}}\|_{\infty}\|\hat{\mathbf{G}}\|_\infty \|\hat{\mathbf{Y}}\|_{\mathcal{H}_2}  + \epsilon^2\|\hat{\mathbf{U}}\|_{\infty} \|\hat{\mathbf{Y}}\|_{\mathcal{H}_2}  }{1 - \epsilon \|\hat{\mathbf{U}}\|_{\infty}} \\
       & \leq  \frac{\|\hat{\mathbf{W}}\|_{\mathcal{H}_2} + \epsilon \|\hat{\mathbf{Y}}\|_{\mathcal{H}_2} (2 + \|\hat{\mathbf{U}}\|_{\infty}\|\hat{\mathbf{G}}\|_\infty )}{1 - \epsilon \|\hat{\mathbf{U}}\|_{\infty}}, \\
    \end{aligned}
\end{equation*}
where the first inequality uses the fact that $ \frac{1}{1 - \epsilon \|\hat{\mathbf{U}}\|_{\infty}} > 1$ and the last inequality uses $\epsilon \|\hat{\mathbf{U}}\|_{\infty} < 1$. 

\subsection{Proof of \cref{prop:LQGcostupperbound}: upper bound on the LQG cost}
\label{app:prop_LQGcostupperbound}

First, if $\|\hat{\mathbf{U}}\|_{\infty} \leq \frac{1}{\epsilon}, \|\mathbf{\Delta}\| < \epsilon$, \cref{lemma:smallgain_ineq} gives the following $\mathcal{H}_{\infty}$ inequalities
$$
    \|(I-\mathbf{\Delta}\hat{\mathbf{U}})^{-1}\|_{\infty} \leq \frac{1}{1 - \epsilon \|\hat{\mathbf{U}}\|_{\infty}}, \qquad \|(I-\hat{\mathbf{U}}\mathbf{\Delta})^{-1}\|_{\infty} \leq \frac{1}{1 - \epsilon \|\hat{\mathbf{U}}\|_{\infty}}.
$$
Using \cref{lemma:HinfH2inequality} in \cref{app:inequalities}, the following inequalities are immediate
\begin{equation} \label{eq:LQGbound_step2}
    \begin{aligned}
    & \|\hat{\mathbf{Y}}(I-\mathbf{\Delta}\hat{\mathbf{U}})^{-1}\|_{\mathcal{H}_2}\leq \|\hat{\mathbf{Y}}\|_{\mathcal{H}_2} \|(I-\mathbf{\Delta}\hat{\mathbf{U}})^{-1}\|_{\infty} \leq \frac{\|\hat{\mathbf{Y}}\|_{\mathcal{H}_2}}{1 - \epsilon \|\hat{\mathbf{U}}\|_{\infty}},\\
     & \|\hat{\mathbf{U}}(I-\mathbf{\Delta}\hat{\mathbf{U}})^{-1}\|_{\mathcal{H}_2}\leq \|\hat{\mathbf{U}}\|_{\mathcal{H}_2} \|(I-\mathbf{\Delta}\hat{\mathbf{U}})^{-1}\|_{\infty} \leq \frac{\|\hat{\mathbf{U}}\|_{\mathcal{H}_2}}{1 - \epsilon \|\hat{\mathbf{U}}\|_{\infty}},\\
      & \|(I-\hat{\mathbf{U}}\mathbf{\Delta})^{-1}\hat{\mathbf{Z}}\|_{\mathcal{H}_2}\leq \|\hat{\mathbf{Z}}\|_{\mathcal{H}_2} \|(I-\hat{\mathbf{U}}\mathbf{\Delta})^{-1}\|_{\infty} \leq \frac{\|\hat{\mathbf{Z}}\|_{\mathcal{H}_2}}{1 - \epsilon \|\hat{\mathbf{U}}\|_{\infty}}. \\
\end{aligned}
\end{equation}

Now, by substituting the upper bounds~\eqref{eq:LQGbound_step2} and~\eqref{eq:LQGbound_step3}  into~\eqref{eq:LQGbound_step1}, we have 
$$
    \begin{aligned}
        &\quad J(\mathbf{G}_\star, \mathbf{K}) \\
        &\leq \frac{1}{1 - \epsilon \|\hat{\mathbf{U}}\|_{\infty}} \sqrt{\|\hat{\mathbf{Y}}\|^2_{\mathcal{H}_2} + \|\hat{\mathbf{U}}\|^2_{{\mathcal{H}}_2} + \|\hat{\mathbf{Z}}\|^2_{\mathcal{H}_2} + \left(\|\hat{\mathbf{W}}\|_{\mathcal{H}_2} + \epsilon \|\hat{\mathbf{Y}}\|_{\mathcal{H}_2} (2 + \|\hat{\mathbf{U}}\|_{\infty}\|\hat{\mathbf{G}}\|_\infty)\right)^2} \\
        & = \frac{1}{1 - \epsilon \|\hat{\mathbf{U}}\|_{\infty}} \sqrt{\left\|\begin{bmatrix} \hat{\mathbf{Y}} & \hat{\mathbf{W}}\\
        \hat{\mathbf{U}}& \hat{\mathbf{Z}}\end{bmatrix}\right\|_{\mathcal{H}_2}^2 + \epsilon\|\hat{\mathbf{W}}\|_{\mathcal{H}_2} \|\hat{\mathbf{Y}}\|_{\mathcal{H}_2} (2 + \|\hat{\mathbf{U}}\|_{\infty}\|\hat{\mathbf{G}}\|_\infty) + \epsilon^2 \|\hat{\mathbf{Y}}\|_{\mathcal{H}_2}^2 (2 + \|\hat{\mathbf{U}}\|_{\infty}\|\hat{\mathbf{G}}\|_\infty)^2 } \\
        &\leq \frac{1}{1 - \epsilon \|\hat{\mathbf{U}}\|_{\infty}} \sqrt{\left\|\begin{bmatrix} \hat{\mathbf{Y}} & \hat{\mathbf{W}}\\
        \hat{\mathbf{U}}& \hat{\mathbf{Z}}\end{bmatrix}\right\|_{\mathcal{H}_2}^2 + \epsilon\|\hat{\mathbf{G}}\|_{\infty} \|\hat{\mathbf{Y}}\|_{\mathcal{H}_2}^2 (2 + \|\hat{\mathbf{U}}\|_{\infty}\|\hat{\mathbf{G}}\|_\infty) + \epsilon^2 \|\hat{\mathbf{Y}}\|_{\mathcal{H}_2}^2 (2 + \|\hat{\mathbf{U}}\|_{\infty}\|\hat{\mathbf{G}}\|_\infty)^2 } \\
        &= \frac{1}{1 - \epsilon \|\hat{\mathbf{U}}\|_{\infty}} \left\|\begin{bmatrix} \sqrt{1 + h(\epsilon,\|\hat{\mathbf{U}}\|_{\infty})} \hat{\mathbf{Y}} & \hat{\mathbf{W}}\\
        \hat{\mathbf{U}}& \hat{\mathbf{Z}}\end{bmatrix}\right\|_{\mathcal{H}_2}, \\
    \end{aligned}
$$
where  $h(\epsilon,\|\hat{\mathbf{U}}\|_{\infty})$ is defined in~\eqref{eq:factor_h}, and the second to the last inequality uses the fact that $\hat{\mathbf{W}} = \hat{\mathbf{Y}}\hat{\mathbf{G}}$ from the affine constraints in~\eqref{eq:robustLQG_YUWZ} and thus $\|\hat{\mathbf{W}}\|_{\mathcal{H}_2} \leq \|\hat{\mathbf{Y}}\|_{\mathcal{H}_2}\|\hat{\mathbf{G}}\|_{\infty}$. This completes the proof.

\subsection{Proof of \cref{theo:mainRobust}: quasi-convex optimization}
\label{app:th_approximation}

Before presenting the proof to the theorem, we review a standard optimization result that is also used in \cite{dean2019sample}:
\begin{myLemma} \label{lemma:quasi}
For functions $f: \mathcal{X} \rightarrow \mathbb{R}$, and $g: \mathcal{X} \rightarrow \mathbb{R}$ and constraint set $C \subseteq \mathcal{X}$, consider
    $$
        \min_{x\in C}  \frac{f(x)}{1 - g(x)}.
    $$
    Assuming that $f(x) \geq 0$ and $0\leq g(x)<1, \forall x \in C$, we have the following equivalence 
    $$
        \min_{x\in C}  \frac{f(x)}{1 - g(x)} = \min_{\gamma \in[0,1)} \frac{1}{1-\gamma}\min_{x\in C} \{f(x) \mid g(x) \leq \gamma\}.
    $$
\end{myLemma}

We are now ready to prove \cref{theo:mainRobust}.
\begin{proof}
    We denote the decision variable as 
    $
        x := (\hat{\mathbf{Y}}, \hat{\mathbf{U}}, \hat{\mathbf{W}}, \hat{\mathbf{Z}}),
    $
    and the functions as 
    $$
    \begin{aligned}
        f(x) :=  \left\|\begin{bmatrix} \sqrt{1 + h(\epsilon,\|\hat{\mathbf{U}}\|_{\infty})} \hat{\mathbf{Y}} & \hat{\mathbf{W}}\\
        \hat{\mathbf{U}}& \hat{\mathbf{Z}}\end{bmatrix}\right\|_{\mathcal{H}_2},  \qquad 
        g(x) := \epsilon \|\hat{\mathbf{U}}\|_{\infty},
    \end{aligned}
    $$
    and the constraint $C$ on $x$ as those in~\eqref{eq:robustLQG_YUWZ}. These functions naturally satisfy $f(x) \geq 0, 0\leq g(x) <1, \forall x \in C$.  Then, by~\eqref{eq:cost_upperbound_nonconvex} in \cref{prop:LQGcostupperbound} and \cref{lemma:quasi}, the robust LQG  problem~\eqref{eq:robustLQG_YUWZ} is upper bounded by
    $$
          \min_{x\in C}  \frac{f(x)}{1 - g(x)} = \min_{\gamma \in[0,1)} \frac{1}{1-\gamma}\min_{x\in C} \{f(x) \mid g(x) \leq \gamma\},
    $$
    where the inner problem reads as 
    $$
        \begin{aligned}
      \min_{\hat{\mathbf{Y}}, \hat{\mathbf{W}}, \hat{\mathbf{U}}, \hat{\mathbf{Z}}} \; &\left\|\begin{bmatrix} \sqrt{1 + h(\epsilon, \|\hat{\mathbf{U}}\|_\infty)} \hat{\mathbf{Y}} & \hat{\mathbf{W}}\\
        \hat{\mathbf{U}}& \hat{\mathbf{Z}}\end{bmatrix}\right\|_{\mathcal{H}_2} \\
        \st \quad &  
        \begin{bmatrix}
        	 I&-\hat{\mathbf{G}}
        	 \end{bmatrix}\begin{bmatrix}
        	 \hat{\mathbf{Y}} & \hat{\mathbf{W}}\\\hat{\mathbf{U}} & \hat{\mathbf{Z}}
        	 \end{bmatrix}=\begin{bmatrix}
        	 I&0
        	 \end{bmatrix}, \\
        &	 \begin{bmatrix}
        	  \hat{\mathbf{Y}} & \hat{\mathbf{W}}\\\hat{\mathbf{U}} & \hat{\mathbf{Z}}
        	 \end{bmatrix}
        	 \begin{bmatrix}
        	 -\hat{\mathbf{G}}\\I
        	 \end{bmatrix}=\begin{bmatrix}
        	 0\\I
        	 \end{bmatrix}, \\
        &	 \hat{\mathbf{Y}}, \hat{\mathbf{W}}, \hat{\mathbf{Z}}\in \mathcal{RH}_\infty, \, \epsilon\|\hat{\mathbf{U}}\|_{\infty} \leq \gamma.
    \end{aligned}
    $$
    Re-scaling $\gamma$ by $\frac{\gamma}{\epsilon}$, introducing an additional constraint $\|\hat{\mathbf{U}}\|_{\infty} \leq \alpha$, and replacing $h(\epsilon, \|\hat{\mathbf{U}}\|_\infty)$ by its upper bound $h(\epsilon, \alpha)$, we arrive at the desired upper bound in~\eqref{eq:robustLQG_convex}.  
\end{proof}

\begin{myRemark}[Quasi-convexity of~\eqref{eq:robustLQG_convex}] \label{remark:quasi_convexity}
It is clear that the following function is convex 
$$
\psi(\gamma, \hat{\mathbf{Y}}, \hat{\mathbf{W}}, \hat{\mathbf{U}},  \hat{\mathbf{Z}}):=\left\|\begin{bmatrix} \sqrt{1 + h(\epsilon, \alpha)} \hat{\mathbf{Y}} & \hat{\mathbf{W}}\\
        \hat{\mathbf{U}}& \hat{\mathbf{Z}}\end{bmatrix}\right\|_{\mathcal{H}_2}.
$$
Denote $C$ as the feasible region for the inner optimization in~\eqref{eq:robustLQG_convex}, which is a convex set. Therefore, the following function $\phi(\gamma)$ obtained by a partial minimization is convex in $\gamma$ 
$$
    \phi(\gamma) := \min_{(\gamma,  \hat{\mathbf{Y}}, \hat{\mathbf{W}}, \hat{\mathbf{U}},  \hat{\mathbf{Z}}) \in C} \psi(\gamma, \hat{\mathbf{Y}}, \hat{\mathbf{W}}, \hat{\mathbf{U}},  \hat{\mathbf{Z}}).
$$
Now, we can see that 
$
    \frac{\phi(\gamma)}{1 - \epsilon \gamma}
$
is quasi-convex in $\gamma$ over $0\leq \gamma < \frac{1}{\epsilon}$, since its sublevel set 
$$
    \left\{\gamma \in \mathbb{R} \left| \frac{\phi(\gamma)}{1 - \epsilon \gamma} \leq a \right.\right\} = \{\gamma \in \mathbb{R} \mid \phi(\gamma) + a \epsilon \gamma \leq a\}
$$
is always convex. 
\end{myRemark}

\section{Suboptimality guarantee in \cref{theo:suboptimality}}
\label{app:theo:suboptimality}

Before proving the theorem, we establish the following technical lemma about existence of a feasible solution to \eqref{eq:robustLQG_convex}.

\begin{myLemma} \label{lemma:suboptimalsolution}
    Consider $(\eta,\tilde{\gamma}, \widetilde{\mathbf{Y}},  \widetilde{\mathbf{U}},  \widetilde{\mathbf{W}},  \widetilde{\mathbf{Z}})$ chosen as follows:
    \begin{equation} \label{eq:suboptimal}
    \begin{aligned}
        \widetilde{\mathbf{Y}} &= \mathbf{Y}_\star (I + \mathbf{\Delta} \mathbf{U}_\star)^{-1}, && 
        \widetilde{\mathbf{U}} = \mathbf{U}_\star (I + \mathbf{\Delta} \mathbf{U}_\star)^{-1}, \\
        \widetilde{\mathbf{W}} &= \mathbf{Y}_\star (I + \mathbf{\Delta} \mathbf{U}_\star)^{-1}(\mathbf{G}_\star - \mathbf{\Delta}), &&
        \widetilde{\mathbf{Z}} = (I +  \mathbf{U}_\star\mathbf{\Delta})^{-1}\mathbf{Z}_\star , \\
        \tilde{\gamma} &= \frac{\sqrt{2}\eta}{\epsilon(1-\eta)},  && \eta = \epsilon \|\mathbf{U}_\star\|_\infty.
    \end{aligned}
\end{equation}
 If $\eta < \frac{1}{5}$, then $(\tilde{\gamma}, \widetilde{\mathbf{Y}},  \widetilde{\mathbf{U}},  \widetilde{\mathbf{W}},  \widetilde{\mathbf{Z}})$ is a feasible solution to~\eqref{eq:robustLQG_convex}. 
\end{myLemma}

\begin{proof}
    Since $\|\mathbf{\Delta}\mathbf{U}_\star\|_\infty \leq \epsilon \|\mathbf{U}_\star\|_\infty < \frac{1}{5}$, the small-gain theorem~\cite[Theorem 9.1]{zhou1996robust} implies that the factors 
    $$
        (1 + \mathbf{\Delta}\mathbf{U}_\star)^{-1}, \qquad  (1 + \mathbf{U}_\star\mathbf{\Delta})^{-1} 
    $$
    are both stable. Therefore, by construction, $\widetilde{\mathbf{Y}},  \widetilde{\mathbf{U}},  \widetilde{\mathbf{W}},  \widetilde{\mathbf{Z}} \in \mathcal{RH}_\infty$. Simple algebraic computations verify that $\widetilde{\mathbf{Y}},  \widetilde{\mathbf{U}},  \widetilde{\mathbf{W}},  \widetilde{\mathbf{Z}}$ satisfy the affine constraints in~\eqref{eq:robustLQG_convex}; indeed, $\widetilde{\mathbf{Y}},  \widetilde{\mathbf{U}},  \widetilde{\mathbf{W}},  \widetilde{\mathbf{Z}}$ are constructed in the way that they are the closed-loop responses when applying $\mathbf{K}_\star$ to the estimated plant $\hat{\mathbf{G}}$. 
    
    Now, it is only required to check the $\mathcal{H}_\infty$ norm inequality
    $$
      \|\widetilde{\mathbf{U}}\|_\infty \leq \tilde{\gamma} = \frac{\sqrt{2}\eta}{\epsilon(1-\eta)} = \frac{\sqrt{2} \|\mathbf{U}_\star\|_\infty}{1-\epsilon \|\mathbf{U}_\star\|_\infty} < \alpha,  
    $$
    which is true by the construction of $\widetilde{\mathbf{U}}$, \emph{i.e.},
    $$
        \|\widetilde{\mathbf{U}}\|_\infty = \|\mathbf{U}_\star (I + \mathbf{\Delta} \mathbf{U}_\star)^{-1}\| \leq \|\mathbf{U}_\star\|_\infty\|(I + \mathbf{\Delta} \mathbf{U}_\star)^{-1}\|_\infty \leq  \frac{ \|\mathbf{U}_\star\|_\infty}{1-\epsilon \|\mathbf{U}_\star\|_\infty}.
    $$
    We finish the proof. 
\end{proof}

Let the optimal solution to~\eqref{eq:robustLQG_convex} be $\gamma_\star, \hat{\mathbf{Y}}_\star, \hat{\mathbf{U}}_\star, \hat{\mathbf{W}}_\star, \hat{\mathbf{Z}}_\star$, and the resulting controller be ${\mathbf{K}} =\hat{\mathbf{U}}_\star\hat{\mathbf{Y}}_\star^{-1}$. 
When applying ${\mathbf{K}} =\hat{\mathbf{U}}_\star\hat{\mathbf{Y}}_\star^{-1}$ to the true plant $\mathbf{G}_\star$, according to \cref{theo:mainRobust}, we achieve the following cost 
\begin{equation} \label{eq:upperbound_s0}
    J(\mathbf{G}_\star, {\mathbf{K}}) \leq  \frac{1}{1 - \epsilon \gamma_\star} \left\|\begin{bmatrix} \sqrt{1 + h(\epsilon,\alpha)} \hat{\mathbf{Y}}_\star & \hat{\mathbf{W}}_\star\\
        \hat{\mathbf{U}}_\star& \hat{\mathbf{Z}}_\star\end{bmatrix}\right\|_{\mathcal{H}_2}.    
\end{equation}

Suppose that the optimal LQG controller to~\eqref{eq:LQG} is $\mathbf{K}_\star$. Applying $\mathbf{K}_\star$ to the true plant $\mathbf{G}_\star$ leads to the optimal closed-loop responses as $\mathbf{Y}_\star, \mathbf{U}_\star, \mathbf{W}_\star, \mathbf{Z}_\star$. We now construct a set of solutions as 
\begin{equation*} 
    \begin{aligned}
        \widetilde{\mathbf{Y}} &= \mathbf{Y}_\star (I + \mathbf{\Delta} \mathbf{U}_\star)^{-1}, && 
        \widetilde{\mathbf{U}} = \mathbf{U}_\star (I + \mathbf{\Delta} \mathbf{U}_\star)^{-1}, \\
        \widetilde{\mathbf{W}} &= \mathbf{Y}_\star (I + \mathbf{\Delta} \mathbf{U}_\star)^{-1}(\mathbf{G}_\star - \mathbf{\Delta}), &&
        \widetilde{\mathbf{Z}} = (I +  \mathbf{U}_\star\mathbf{\Delta})^{-1}\mathbf{Z}_\star , \\
        \tilde{\gamma} &= \frac{\sqrt{2}\eta}{\epsilon(1-\eta)},  && \eta = \epsilon \|\mathbf{U}_\star\|_\infty.
    \end{aligned}
\end{equation*}
As verified in \cref{lemma:suboptimalsolution}, when $\eta \leq \frac{1}{5}$, the set of solutions $\tilde{\gamma}, \widetilde{\mathbf{Y}},  \widetilde{\mathbf{U}},  \widetilde{\mathbf{W}},  \widetilde{\mathbf{Z}}$ constructed above is feasible to~\eqref{eq:robustLQG_convex} and thus suboptimal. Therefore, by optimality, we have
\begin{equation} \label{eq:upperbound_s1}
    \frac{1}{1 - \epsilon \gamma_\star} \left\|\begin{bmatrix} \sqrt{1 + h(\epsilon,\alpha)} \hat{\mathbf{Y}}_\star & \hat{\mathbf{W}}_\star\\
        \hat{\mathbf{U}}_\star& \hat{\mathbf{Z}}_\star\end{bmatrix}\right\|_{\mathcal{H}_2} \leq \frac{1}{1 - \epsilon\tilde{\gamma}} \left\|\begin{bmatrix} \sqrt{1 + h(\epsilon,\alpha} \widetilde{\mathbf{Y}} & \widetilde{\mathbf{W}}\\
        \widetilde{\mathbf{U}}& \widetilde{\mathbf{Z}}\end{bmatrix}\right\|_{\mathcal{H}_2}. 
\end{equation}
Furthermore, similar to~\eqref{eq:LQGbound_step2} and~\eqref{eq:LQGbound_step3}, we have 
\begin{equation} \label{eq:upperbound_s2}
\begin{aligned}
    &\quad \left\|\begin{bmatrix} \sqrt{1 + h(\epsilon,\alpha)} \widetilde{\mathbf{Y}} & \widetilde{\mathbf{W}}\\
        \widetilde{\mathbf{U}}& \widetilde{\mathbf{Z}}\end{bmatrix}\right\|_{\mathcal{H}_2} \\
        &= \sqrt{(1 + h(\epsilon,\alpha)\|\widetilde{\mathbf{Y}}\|_{\mathcal{H}_2}^2 + \|\widetilde{\mathbf{U}}\|_{\mathcal{H}_2}^2 + \|\widetilde{\mathbf{W}}\|_{\mathcal{H}_2}^2 + \|\widetilde{\mathbf{Z}}\|_{\mathcal{H}_2}^2} \\
        & \leq \frac{1}{1 - \epsilon \|\mathbf{U}_\star\|_\infty}\sqrt{\left\|\begin{bmatrix} {\mathbf{Y}}_\star & {\mathbf{W}}_\star\\
        {\mathbf{U}}_\star& {\mathbf{Z}}_\star\end{bmatrix}\right\|_{\mathcal{H}_2}^2 + \left(h(\epsilon,\alpha)+g(\epsilon, \|\mathbf{U}_\star\|_\infty)\right) \|\mathbf{Y}_\star\|^2_{\mathcal{H}_2}},
\end{aligned}
\end{equation}
where $g(\epsilon, \|\mathbf{U}_\star\|_\infty)$ is defined in~\eqref{eq:constant_g}. Combining~\eqref{eq:upperbound_s0}---\eqref{eq:upperbound_s2}, we have 
\begin{equation*}
    \begin{aligned}
         &\quad J(\mathbf{G}_\star, {\mathbf{K}}) \\
         &\leq  \frac{1}{1 - \epsilon\tilde{\gamma}} \frac{1}{1 - \epsilon \|\mathbf{U}_\star\|_\infty}\sqrt{\left\|\begin{bmatrix} {\mathbf{Y}}_\star & {\mathbf{W}}_\star\\
        {\mathbf{U}}_\star& {\mathbf{Z}}_\star\end{bmatrix}\right\|_{\mathcal{H}_2}^2 + \left(h(\epsilon,\alpha)+g(\epsilon, \|\mathbf{U}_\star\|_\infty)\right) \|\mathbf{Y}_\star\|^2_{\mathcal{H}_2}} \\
        &= \frac{1}{1 -\frac{\sqrt{2}\eta}{1-\eta}} \frac{1}{1 - \eta}\sqrt{J(\mathbf{G}_\star, \mathbf{K}_\star)^2  + \left(h(\epsilon,\alpha)+g(\epsilon, \|\mathbf{U}_\star\|_\infty)\right) \|\mathbf{Y}_\star\|^2_{\mathcal{H}_2}} \\
         &= \frac{1}{1 -(1+\sqrt{2})\eta} \sqrt{J(\mathbf{G}_\star, \mathbf{K}_\star)^2  + \left(h(\epsilon,\alpha)+g(\epsilon, \|\mathbf{U}_\star\|_\infty)\right) \|\mathbf{Y}_\star\|^2_{\mathcal{H}_2}}. \\
    \end{aligned}
\end{equation*}
This implies that 
\begin{equation} \label{eq:upperboudn_s3}
    \begin{aligned}
         &\quad \frac{J(\mathbf{G}_\star, {\mathbf{K}})^2 - J(\mathbf{G}_\star, {\mathbf{K}}_\star)^2}{J(\mathbf{G}_\star, {\mathbf{K}}_\star)^2} \\
         &\leq \frac{J(\mathbf{G}_\star, \mathbf{K}_\star)^2  + \left(h(\epsilon,\alpha)+g(\epsilon, \|\mathbf{U}_\star\|_\infty)\right) \|\mathbf{Y}_\star\|^2_{\mathcal{H}_2}}{(1 -(1+\sqrt{2})\eta)^2J(\mathbf{G}_\star, {\mathbf{K}}_\star)^2}   - 1 \\
         &= \frac{2(1+\sqrt{2}) - (1+\sqrt{2})^2\eta}{(1 -(1+\sqrt{2})\eta)^2}\eta + \left(h(\epsilon,\alpha)+g(\epsilon, \|\mathbf{U}_\star\|_\infty)\right) \frac{ \|\mathbf{Y}_\star\|^2_{\mathcal{H}_2}}{J(\mathbf{G}_\star, {\mathbf{K}}_\star)^2} \\
         &\leq 20 \eta + \left(h(\epsilon,\alpha)+g(\epsilon, \|\mathbf{U}_\star\|_\infty)\right) \frac{ \|\mathbf{Y}_\star\|^2_{\mathcal{H}_2}}{J(\mathbf{G}_\star, {\mathbf{K}}_\star)^2}\,,
    \end{aligned}
\end{equation}
where the last inequality follows the fact that condition $\eta < \frac{1}{5} < \frac{1}{2+2\sqrt{2}}$ leads to
$$
\frac{2(1+\sqrt{2}) - (1+\sqrt{2})^2\eta}{(1 -(1+\sqrt{2})\eta)^2} \leq \frac{2(1+\sqrt{2})}{(1 -(1+\sqrt{2})\eta)^2} \leq  \frac{2(1+\sqrt{2})}{(1 -\frac{1}{2})^2} < 20.
$$
%
Further, it is immediate to see   
$$ 
    \|\mathbf{Y}_\star\|^2_{\mathcal{H}_2} \leq J(\mathbf{G}_\star, {\mathbf{K}}_\star)^2.
$$
Substituting this inequality 
into~\eqref{eq:upperboudn_s3} completes the proof of~\eqref{eq:suboptimality}.

\section{Non-asymptotic system identification bound}
	\label{app:theo_FIR_identification}

In this section, we discuss a recent non-asymptotic system identification result from~\cite{oymak2019non}, which only requires a single finite-length input-output trajectory. Other non-asymptotic estimation procedures, e.g.,~\cite{tu2017non,sarkar2019finite,zheng2020non}, are also suitable in our robust controller synthesis procedure. We note that \cref{pr:prop_norm_inequalities} connects a spectral radius on Markov parameters with its $\mathcal{H}_{\infty}$ norm, which might be of independent interest.

	\subsection{Estimation Procedure} \label{app:theo_FIR_identification_procedure}

The algorithm proposed in~\cite{oymak2019non} excites the plant~\eqref{eq:dynamic} with random Gaussian inputs $u_t\sim \mathcal{N}(0,\sigma_u^2I)$, and collect the output $y_t$ for $\bar{N}$ consecutive steps. This single trajectory of input-output data is denoted as $\{y_t, u_t\}, t = 0, \ldots, \bar{N}$. According to~\cite{oymak2019non}, we subdivide the data into $N$ subsequences, each of length $T$, where $\bar{N} = T + N - 1$. 

Upon organizing the data $u_t$ and the noise $w_t$ as 
$$
\begin{aligned}
    \bar{u}_t &= \begin{bmatrix} u_t^\tr & u_{t-1}^\tr & \ldots & u_{t - T + 1}^\tr  \end{bmatrix}^\tr \in \mathbb{R}^{Tm}, \\ 
    \bar{w}_t &= \begin{bmatrix} w_t^\tr & w_{t-1}^\tr & \ldots & w_{t - T + 1}^\tr  \end{bmatrix}^\tr \in \mathbb{R}^{Tm},
\end{aligned}
$$
simple recursive calculations show that 
\begin{equation} \label{eq:LSrecursion}
    y_t = G_\star \bar{u}_t + G_\star \bar{w}_t + v_t+e_t, \qquad t = T, \ldots, \bar{N},
\end{equation}
where $e_t = C_\star A^{T}_\star x_{t-T+1}$ will be sufficiently small if $T$ is large enough.  Now, we use the following ordinary least-squares (OLS) estimator
\begin{equation} \label{eq:LS_v1}
    \hat{G} \in \argmin_{G} \; \sum_{t=T}^{\bar{N}} \|y_t - G \bar{u}_t\|_2^2.
\end{equation}
The least-squares solution to~\eqref{eq:LS_v1} is given by  $\hat{G}=(({U}^\tr U)^{-1}U^\tr {Y})^\tr$, where 
\begin{equation*}
\begin{aligned}
    Y &= \begin{bmatrix}y_T & y_{T+1} & \cdots & y_{\bar{N}} \end{bmatrix}^\tr \in \mathbb{R}^{N \times p},  \\
    U &= \begin{bmatrix} \bar{u}_{T} & \bar{u}_{T+1} & \ldots & \bar{u}_{\bar{N}}\end{bmatrix}^\tr \in \mathbb{R}^{N \times Tm}.
\end{aligned}
\end{equation*}
From the OLS solution $\hat{G} = \begin{bmatrix} \hat{G}_0 & \hat{G}_1 & \ldots & \hat{G}_{T-1} \end{bmatrix}$, we form the estimated FIR plant in~\eqref{eq:FIRestimation}. 

	\subsection{Non-symptomatic Bounds} \label{app:theo_FIR_identification_bound}
	
	We first recall the following non-asymptotic result in term of the spectral norm of $\|\hat{G} - G_\star\|$ that scales as $\mathcal{O}(\frac{1}{\sqrt{N}})$~\citep{oymak2019non}.

\begin{myTheorem}[{\!\!\cite[Theorem~3.2]{oymak2019non}}]
\label{th:bound_Markov}
Suppose $\rho(A_\star) < 1$ and $$N\geq cTm \log^2{(2Tm)}\log^2{(2Nm)}\,,$$ where $c$ is a universal constant. We collect a single trajectory $(y_t, u_t)$ until time $\bar{N}=N+T-1$ using random inputs $u_t \sim \mathcal{N}(0, \sigma_u^2 I)$. Then, with high probability, the least-squares estimation $\hat{G}$ from~\eqref{eq:LS_v1} satisfies
\begin{equation} \label{eq:bound_markov}
    \|\hat{G} - G_\star\| \leq  \frac{R_w + R_v + R_e}{\sigma_u} \frac{1}{\sqrt{N}},
\end{equation}
where $R_w, R_v, R_e$ are problem-dependent constants corresponding to the noise terms $w_t, v_t, e_t$ in~\eqref{eq:LSrecursion}. 
\end{myTheorem}

For the robust controller synthesis in \cref{section:robustIOP}, we need to bound the $\mathcal{H}_\infty$ norm of $\mathbf{G}_\star - \hat{\mathbf{G}}$, for which we have 
\begin{equation} \label{eq:EstimationErrorHinf_s0}
\begin{aligned}
    \|\mathbf{G}_\star - \hat{\mathbf{G}}\|_\infty &= \left\| \sum_{t=0}^{T-1} \left(G_\star(k) - \hat{G}_k\right) \frac{1}{z^k} + \sum_{k=T}^{\infty} G_\star(k) \frac{1}{z^k}  \right\|_\infty \\
    & \leq  \underbrace{\left\| \sum_{t=0}^{T-1} \left(G_\star(k) - \hat{G}_k\right) \frac{1}{z^k}\right\|_\infty}_{\text{FIR estimation error}} \; + \;\; \underbrace{\left\| \sum_{k=T}^{\infty} G_\star(k) \frac{1}{z^k}  \right\|_\infty}_{\text{FIR truncation error}}, 
\end{aligned}
\end{equation}
where the inequality follows the standard norm triangular inequality; see~\eqref{eq:tr_ineq} in \cref{app:inequalities}.  To bound the terms in~\eqref{eq:EstimationErrorHinf_s0}, we present the following $\mathcal{H}_\infty$ norm bounds for FIR transfer matrices. 

\begin{myProposition}
	\label{pr:prop_norm_inequalities}
	Given an FIR transfer matrix $\mathbf{G}=\sum_{i=0}^{T-1} G_i \frac{1}{z^i}$, we  have
    \begin{equation}
    \label{eq:block_matrix_ineq}
        \|{\mathbf{G}}\|_{\infty} \leq \sqrt{T} \left\|\begin{bmatrix} G_0 & G_1 & \ldots & G_{T-1} \end{bmatrix}\right\|,
    \end{equation}
    and each spectral component satisfies 
    \begin{equation} \label{eq:single_FIR_ineq}
        \left\|G_i \frac{1}{z^i} \right\|_{\infty} \leq \|G_i\|, \qquad \forall i = 0, \ldots, T-1.
    \end{equation}
\end{myProposition}
\begin{proof}
Given any signal $\mathbf{u} = \{u_0, u_1, \ldots, u_k, \ldots\}$ with finite energy, \emph{i.e.}, $\sum_{k=0}^\infty u_k^\tr u_k < \infty$, we construct
$
    u_{-1} = 0, \; \ldots, \; u_{1-T} = 0,
$
for notational convenience. 
Then, the output signal $\mathbf{y} = \mathbf{G}\mathbf{u} = \{y_0, y_1, \ldots, y_k, \ldots\}$ can be written as
$$
\begin{aligned}
    y_k = \begin{bmatrix} G_0 & G_1 & \ldots & G_{T-1} \end{bmatrix} \begin{bmatrix} u_k \\  u_{k-1} \\ \vdots \\  u_{k - T + 1}  \end{bmatrix}, \quad  k = 0, 1, 2, \ldots  \\
\end{aligned}
$$
By the definition of $\mathcal{H}_\infty$ norm, we have 
\begin{equation} \label{eq:FIRHinf_s1}
    \begin{aligned}
        \|\mathbf{G}\|_\infty^2 &= \max_{\sum_{k=0}^\infty u_k^\tr u_k = 1}   \displaystyle \sum_{k=0}^\infty y_k^\tr y_k \\
        &=  \max_{\sum_{k=0}^\infty u_k^\tr u_k = 1} \sum_{k=0}^{\infty} \left\|\begin{bmatrix} G_0 & G_1 & \ldots & G_{T-1} \end{bmatrix} \begin{bmatrix} u_k \\  u_{k-1} \\ \vdots \\  u_{k - T + 1}  \end{bmatrix}\right\|^2 \\
        &\leq \left\|\begin{bmatrix} G_0 & G_1 & \ldots & G_{T-1} \end{bmatrix}\right\|^2 \max_{\sum_{k=0}^\infty u_k^\tr u_k = 1} \sum_{k=0}^{\infty} \left\|\begin{bmatrix} u_k \\  u_{k-1} \\ \vdots \\  u_{k - T + 1}  \end{bmatrix}\right\|^2, \\
    \end{aligned}
\end{equation}
where the last inequality follows from the sub-multiplicative property
$
    \|Au\| \leq \|A\| \|u\|.
$
Now, it is straightforward to verify that 
$$
    \sum_{k=0}^{\infty} (u_k^\tr u_k +  u_{k-1}^\tr u_{k-1} +  \ldots +  u_{k - T + 1}^\tr u_{k - T + 1}) = T \sum_{k=0}^\infty u_k^\tr u_k.
$$
Substituting this identity to~\eqref{eq:FIRHinf_s1} leads to~\eqref{eq:block_matrix_ineq}.
Similarly, we have 
$$
    \begin{aligned}
         \left\|G_i \frac{1}{z^i} \right\|_{\infty}^2   &= \max_{\sum_{k=0}^\infty u_k^\tr u_k = 1}   \displaystyle \sum_{k=0}^\infty y_k^\tr y_k  \\
        &=  \max_{\sum_{k=0}^\infty u_k^\tr u_k = 1} \sum_{k=0}^{\infty} \left\|G_i u_{k-i}\right\|^2 \\
        &\leq \left\|G_i\right\|^2 \max_{\sum_{k=0}^\infty u_k^\tr u_k = 1} \sum_{k=0}^{\infty} \|u_{k-i}\|^2 \\
        &= \left\|G_i\right\|^2. 
    \end{aligned}
$$
Thus, we complete the proof.
\end{proof}

The inequality~\eqref{eq:block_matrix_ineq} seems not standard in the control literature, and a more standard $\mathcal{H}_\infty$ inequality is 
\begin{equation} \label{eq:block_matrix_ineq_v2}
    \left\|\sum_{i=0}^{T-1} G_i \frac{1}{z^i} \right\|_{\infty} \leq \sum_{i=0}^{T-1}  \left\|G_i \frac{1}{z^i} \right\|_{\infty} \leq \sum_{i=0}^{T-1} \|G_i\|,
\end{equation}
where we have applied~\eqref{eq:single_FIR_ineq} and the norm triangular inequality~\eqref{eq:tr_ineq}. A more general version of~\eqref{eq:block_matrix_ineq_v2} is available in~\cite[Theorem 4.5]{zhou1996robust}.  
However, it is clear that the inequality~\eqref{eq:block_matrix_ineq} suits better for our purpose of the plant estimation, as the quantity $\left\|\begin{bmatrix} G_0 & G_1 & \ldots & G_{T-1} \end{bmatrix}\right\|$ appears more naturally in the OLS estimation error analysis; see~\eqref{eq:bound_markov} and~\eqref{eq:EstimationErrorHinf_s0}. 

\begin{myRemark}
    We note that~\eqref{eq:block_matrix_ineq} and~\eqref{eq:block_matrix_ineq_v2}  are not comparable, but they have the same upper bound as 
	$$
	\begin{aligned}
	\sum_{i=0}^{T-1} \|G_i\| &\leq \sqrt{T \sum_{i=0}^{T-1} \|G_i\|^2}, \\
	 \sqrt{T} \left\|\begin{bmatrix} G_0 & G_1 & \ldots & G_{T-1} \end{bmatrix}\right\|  & \leq 	 \sqrt{T} \left\|\begin{bmatrix} \|G_0\| & \|G_1\| & \ldots & \|G_{T-1}\| \end{bmatrix}\right\| = \sqrt{T \sum_{i=0}^{T-1} \|G_i\|^2}\,,
	\end{aligned}
	$$
	where the first one is from the  Cauchy-Schwarz inequality and the second one is based on~\cite[Lemma 2.10]{zhou1996robust}. For example, consider
	$$
	    \mathbf{G}_1 = \begin{bmatrix} 1 & 2 \\ 2 & 1 \end{bmatrix} + \begin{bmatrix} 1 & 0 \\ 0 & 1 \end{bmatrix} \frac{1}{z}, \qquad   \mathbf{G}_2 = \begin{bmatrix} 1 & 2 \\ 1 & 1 \end{bmatrix} + \begin{bmatrix} 1 & 0 \\ 2 & 1 \end{bmatrix} \frac{1}{z}.
	$$
	The upper bound on $\|\mathbf{G}_1\|_\infty$ from~\eqref{eq:block_matrix_ineq_v2} is $4$, and that from~\eqref{eq:block_matrix_ineq} is $4.48$. Instead, for $\|\mathbf{G}_2\|_\infty$,~\eqref{eq:block_matrix_ineq_v2} returns $5.03$ and~\eqref{eq:block_matrix_ineq} gives $4.80$. In addition, the bounds~\eqref{eq:block_matrix_ineq} and~\eqref{eq:block_matrix_ineq_v2}  are true for any FIR transfer matrices. 
\end{myRemark}

We are now ready to combine the above results to complete the proof of \cref{co:FIRestimation}, that is, we derive the $\mathcal{H}_\infty$ norm bound on the FIR estimation error from the OLS estimator~\eqref{eq:LS_v1}.

\begin{proof}
Applying \cref{pr:prop_norm_inequalities} to~\eqref{eq:EstimationErrorHinf_s0}, we have
\begin{equation}
    \begin{aligned}
         \|\mathbf{G}_\star - \hat{\mathbf{G}}\|_\infty     &\leq \sqrt{T} \| G_\star - \hat{G}\| + \sum_{k=T}^\infty \|G_\star(k)\|\,. \label{eq:intermediate_bound}
    \end{aligned}    
\end{equation}
Now, we have the spectral component $G_\star(k) = C_\star A_\star^k B_\star$. From the sub-multiplicative property of induced norms, we have
\begin{equation*}
   \|C_\star A_\star^k B_\star\| \leq\|C_\star\| \|B_\star\|\|A_\star^k\| \leq  \norm{C^\star}\norm{B^\star} \Phi(A^\star)\rho(A_\star)^k\,,
\end{equation*}
Then, using the geometric series, it is straightforward to see that
\begin{equation}
    \label{eq:bound_exponentialdecay}
    \sum_{k=T}^\infty \|G_\star(k)\| \leq  \Phi(A_\star) \norm{C_\star} \norm{B_\star} \frac{\rho(A_\star)^T}{1-\rho(A_\star)}\,.
\end{equation}
We complete the proof by applying \cref{th:bound_Markov} and \eqref{eq:bound_exponentialdecay} to \eqref{eq:intermediate_bound}.
\end{proof}

\section{Derivation of the Cost in~\eqref{eq:LQG} as an $\mathcal{H}_2$ norm} \label{app:h2norm}

Here, we consider the frequency-domain description of the infinite horizon LQG problem~\eqref{eq:LQG}. In particular, we show how it can be recast as an equivalent $\mathcal{H}_2$ optimal control in terms of the system responses, defined in~\eqref{eq:LQG_YUWZ}. 

First, recall that stable and achievable closed-loop responses $(\mathbf{Y}, \mathbf{U}, \mathbf{W}, \mathbf{Z})$ are characterized by~\eqref{eq:affineYUWZ}. They describe the closed-loop maps from disturbance $\mathbf{w}, \mathbf{v}$ to the output and control action $(\mathbf{y}, \mathbf{u})$, achieved by a internally stabilizing controller $\mathbf{K}$, \emph{i.e.}, 
$$
    \begin{aligned}
         \mathbf{y} &= \mathbf{Y} \mathbf{v} + \mathbf{W} \mathbf{w}, \\
         \mathbf{u} &= \mathbf{U} \mathbf{v} + \mathbf{Z} \mathbf{w}.
    \end{aligned}
$$
Since $(\mathbf{Y}, \mathbf{U}, \mathbf{W}, \mathbf{Z})$ are stable, they can be equivalently written into the form of impulse responses 
$$
    \begin{aligned}
         \mathbf{Y}(z) = \sum_{t=0}^{\infty} Y_t \frac{1}{z^t}, \qquad        \mathbf{U}(z) = \sum_{t=0}^{\infty} U_t \frac{1}{z^t}, \quad          \mathbf{W}(z) = \sum_{t=0}^{\infty} W_t \frac{1}{z^t}, \qquad        \mathbf{Z}(z) = \sum_{t=0}^{\infty} Z_t \frac{1}{z^t},
    \end{aligned}
$$
where $Y_t \in \mathbb{R}^{p \times p}, U_t \in \mathbb{R}^{p \times m}, W_t \in \mathbb{R}^{m \times m}, Z_t \in \mathbb{R}^{p \times m}$ are the $t$th spectral components. Then, the closed-loop responses $y_t, u_t$ can be described as
$$
    \begin{aligned}
         y_t = \sum_{k=0}^{t} \left(Y_{k}v_{t-k} + W_{k}w_{t-k}\right), \qquad
         u_t= \sum_{k=0}^{t} \left(U_{k}v_{t-k} + Z_{k}w_{t-k}\right).
    \end{aligned}
$$
For i.i.d. disturbances distributed as $w_t \sim \mathcal{N}(0, \sigma_w^2 I_m), v_t \sim \mathcal{N}(0, \sigma_v^2 I_p$, we have
$$
\begin{aligned}
    \mathbb{E}\left[ y_t^\tr Q y_t \right] &= \mathbb{E}\left[\left(\sum_{k=0}^t \left(Y_{k}v_{t-k} + W_{k}w_{t-k}\right)\right)^\tr Q\left(\sum_{k=0}^t \left(Y_{k}v_{t-k} + W_{k}w_{t-k}\right) \right)\right] \\
    &= \sigma_v^2 \sum_{k=0}^t \mathbb{E}\left[\text{tr}\left(Y_k^\tr Q Y_k\right)\right] + \sigma_w^2 \sum_{k=0}^t \mathbb{E}\left[\text{tr}\left(W_k^\tr Q W_k\right)\right], \\
    \mathbb{E}\left[ u_t^\tr R u_t \right] &= \mathbb{E}\left[\left(\sum_{k=0}^t  \left(U_{k}v_{t-k} + Z_{k}w_{t-k}\right)\right)^\tr R\left(\sum_{k=0}^t \left(U_{k}v_{t-k} + Z_{k}w_{t-k}\right)\right)\right] \\
    &= \sigma_v^2 \sum_{k=0}^t \mathbb{E}\left[\text{tr}\left(U_k^\tr R U_k\right)\right] + \sigma_w^2 \sum_{k=0}^t \mathbb{E}\left[\text{tr}\left(Z_k^\tr R Z_k\right)\right].
\end{aligned}
$$

We then have
$$
\begin{aligned}
    &\lim_{T\rightarrow \infty }\mathbb{E}\left[\frac{1}{T}\sum_{t=0}^T\left(y_t^\tr Qy_t + u_t R u_t\right)\right] \\
    = &  \mathbb{E}\left[y_{\infty}^\tr Qy_{\infty} + u_{\infty} R u_{\infty}\right] \\
     = &\sum_{k=0}^\infty \mathbb{E}\left[\sigma_v^2\text{tr}\left(Y_k^\tr Q Y_k\right) + \sigma_w^2\text{tr}\left(W_k^\tr Q W_k\right)+ \sigma_v^2\text{tr}\left(U_k^\tr R U_k\right) + \sigma_w^2\text{tr}\left(Z_k^\tr R Z_k\right)\right] \\
    =&\sum_{k=0}^{\infty}\left\|\begin{bmatrix} Q^{\frac{1}{2}} & \\ & R^{\frac{1}{2}} \end{bmatrix} \begin{bmatrix} {Y}_k & W_k\\ {U}_k & Z_k \end{bmatrix}\begin{bmatrix} \sigma_vI_p & \\ & \sigma_w I_m \end{bmatrix} \right\|^2_{F} \\
    = &\frac{1}{2\pi} \int_{-\pi}^{\pi} \left\|\begin{bmatrix} Q^{\frac{1}{2}} & \\ & R^{\frac{1}{2}} \end{bmatrix} \begin{bmatrix} \mathbf{Y}(e^{j\theta}) & \mathbf{W}(e^{j\theta}) \\ \mathbf{U}(e^{j\theta}) & \mathbf{Z}(e^{j\theta})\end{bmatrix}\begin{bmatrix} \sigma_vI_p & \\ & \sigma_wI_m \end{bmatrix}\right\|^2_{F} d\theta  \\
    =&\left\|\begin{bmatrix} Q^{\frac{1}{2}} & \\ & R^{\frac{1}{2}} \end{bmatrix} \begin{bmatrix} \mathbf{Y} & \mathbf{W} \\ \mathbf{U} & \mathbf{Z}\end{bmatrix}\begin{bmatrix} \sigma_vI_p & \\ & \sigma_wI_m \end{bmatrix}\right\|^2_{\mathcal{H}_2},
\end{aligned}
$$
where the last equality is the definition of $\mathcal{H}_2$ norm, and the second to last equality is the due to Parseval's Theorem.

}

\end{document}